\documentclass[10pt]{amsart}
\usepackage{amsmath,amssymb,latexsym,cancel,rotating}
\usepackage{graphicx,mathrsfs,color,fancyhdr,amsthm}
\usepackage{tikz}
\usepackage{tikz-cd}
\usepackage{geometry}
\usepackage{hyperref}
\usepackage{cite}
\usepackage{setspace}
\newtheorem{theorem}{Theorem}[section]
\newtheorem{lemma}[theorem]{Lemma}
\newtheorem{corollary}[theorem]{Corollary}
\newtheorem{definition}[theorem]{Definition}
\newtheorem{proposition}[theorem]{Proposition}

\newtheorem{remark}[theorem]{Remark}

 \def\cD{{\mathcal D}}

\def\bbN{{\mathbb N}}  \def\bbZ{{\mathbb Z}}  \def\bbQ{{\mathbb Q}}
    \def\bbF{{\mathbb F}}

  \def\leq{\leqslant}  \def\geq{\geqslant}

\def\Hom{\mbox{\rm Hom}}  
       
\def\Ext{\mbox{\rm Ext}}   
\def\dim{\mbox{\rm dim}}

\def\mod{\mbox{\rm mod}}

\def\bfV{{\mathbf V}}

\def\bfV{\mathbf{V}}
\def\bfW{\mathbf{W}}
\def\bfE{\mathbf{E}}
\def\bfG{\mathbf{G}}

\def\Ind{\mathbf{Ind}}
\def\Res{\mathbf{Res}}
\def\ind{\mathbf{ind}}
\def\res{\mathbf{res}}

\newcommand\bbq{{\mathbb Q}}

\newcommand\mk{{\mathcal{K}}}

\newcommand\mo{{\mathcal{O}}}
\newcommand\mP{{\mathcal{P}}}
\newcommand\mq{{\mathcal{Q}}}

\newcommand\ml{{\mathcal{L}}}

\begin{document}

\title[]{Construction of PBW and canonical bases in the folding extension algebras for symmetrizable types}

\author{Yumeng Wu}

\address{Beijing International Center for Mathematical Research, Beijing 100871, P. R. China}
\email{2506397175@pku.edu.cn (Y. Wu)}



\bibliographystyle{abbrv}

\begin{abstract}
We study extension algebras arising from Lusztig's construction of perverse sheaves on quiver representation spaces equipped with an admissible automorphism $a$, leading to folding Khovanov--Lauda--Rouquier (KLR) algebras. Extending the method of M. Varagnolo and E. Vasserot to Lusztig's symmetrizable setting, we construct the skew group algebra $R(\nu)\langle\mathbf{a}\rangle$, where $\mathbf{a}$ is induced by the admissible automorphism $a$, and show that $\Ext^{\bullet}(\ ,L)$ preserves the signed basis and the crystal structure. In this framework, for finite types, we define standard and dual standard modules, analyze their behavior under dualities, and study the resulting Grothendieck groups. We prove that the transition matrix from the classes of indecomposable projective modules, which give the canonical basis, to the dual standard modules, which give the PBW basis, is upper triangular with diagonal entries equal to one. This provides a geometric categorification of these bases and of their relation in the symmetrizable setting.
\end{abstract}

\maketitle

\setcounter{tocdepth}{2}\tableofcontents
\begin{spacing}{1.5}
\section{Introduction}
Geometric constructions over representations of quivers play a central role in the theory of quantum groups. One of the foundational developments is Lusztig's construction of canonical bases via categories of perverse sheaves on moduli spaces of quiver representations~\cite{lusztig1990canonical, lusztig2010introduction}. Quiver Hecke algebras were introduced independently by Khovanov and Lauda~\cite{KLR} and by Rouquier~\cite{rouquier20082kacmoodyalgebras,doi:10.1142/S1005386712000247}. Khovanov and Lauda used categories of finitely generated graded projective modules over these algebras to categorify the integral form of the negative half of the corresponding quantum group, while Rouquier developed these algebras within the framework of $2$-Kac-Moody algebras and their $2$-representations. Throughout this paper, we refer to quiver Hecke algebras as Khovanov--Lauda--Rouquier (KLR) algebras. In the symmetric case, Varagnolo and Vasserot gave a geometric realization of these algebras by identifying KLR algebras with the graded equivariant extension algebras of the corresponding Lusztig complexes on quiver representation varieties~\cite{VaragnoloVasserot+2011+67+100}. This realization provides a direct link between the representation theory of KLR algebras and Lusztig's geometric construction of the canonical basis. These constructions provide deep insight into the structure of quantum enveloping algebras and establish a powerful bridge between geometry and algebra.

In this paper, we study the folding extension algebra associated with Lusztig sheaves complexes over the moduli space of a quiver equipped with an admissible automorphism. Motivated by McNamara's folding of Khovanov--Lauda--Rouquier (KLR) algebras~\cite{https://doi.org/10.1112/jlms.12218}, which is further studied in \cite{MA202460} and Kato's approach to extension algebras~\cite{cd088214-009e-3813-9a86-97f9a299a24c}, we construct and analyze a category of modules over the skew group algebra $R(\nu)\langle\mathbf{a}\rangle$, which arises from $R(\nu)$ the Ext-algebra  of Lusztig sheaves together with the periodic functor induced by the automorphism. We show that this category is equivalent to the category of $R(\nu)$-modules endowed with a compatible action of the automorphism. We also prove that bialgebraic structure on the Grothendieck group of projective modules is induced by Lusztig's induction functor and restriction functor, compatibly with the periodic structure.

A central part of the paper is the comparison of two natural bases in the Grothendieck groups associated with projective modules: the PBW basis, given by dual standard modules, and the canonical basis, given by indecomposable projective modules. We construct both standard and dual standard modules in the folded setting and prove that the transition matrix relating these bases is upper triangular with diagonal entries equal to one. This mirrors the classical relation between canonical and PBW bases and shows that the folding extension algebra retains the expected structural features.

Our results also provide a categorical realization of the bar involution and of the relevant duality structures, extending known results from the symmetric case to the symmetrizable setting. In addition, we prove a pairing identity expressing the duality between standard and dual standard modules, thereby establishing a precise correspondence between the PBW and canonical bases in the Grothendieck group.

The paper is organized as follows. In Section 2, we recall Lusztig's theory for quiver representations with an admissible automorphism and introduce the periodic triangulated category $\widetilde{\mathcal{Q}}_V$. In Section 3, we define the folding extension algebra $R(\nu)\langle\mathbf{a}\rangle$ in Subsection 3.1, study its module categories and derive the explicit descriptions of the projective modules in Subsection 3.2, and construct the PBW and canonical bases within the Grothendieck group in Subsection 3.3. We then compare these bases and determine the transition matrix between them.

\section{Lusztig sheaves complexes over representation space of a quiver with an admissible automorphism}
In this section, we introduce the representation space of a quiver together with the action of $a$. Our main goal is to study the periodic category of semisimple perverse sheaves on this space, together with the corresponding $a$-equivariant morphism. Throughout the paper, the ground field is the closure of the finite field $k=\overline{\bbF_q}$.

\begin{definition}
A quiver $Q$ consists of the following data:
\begin{itemize}
    \item A set of vertices $I = \{1,2,\ldots,n\}$.
    \item A set of arrows $H$, together with two maps $s, t : H \to I$ such that for each arrow $h \in H$, $s(h)$ denotes its source and $t(h)$ denotes its target.
\end{itemize}
\end{definition}

\begin{definition}
   An automorphism $a$ of the quiver $Q = (I, H, s, t)$ is called \emph{admissible} if it satisfies the following conditions:
\begin{itemize}
    \item $a$ is an automorphism of both the set $I$ and the set $H$;
    \item For any $h \in H$, one has
    \[
    a(s(h)) = s(a(h)) \quad \text{and} \quad a(t(h)) = t(a(h));
    \]
    \item For each $a$-orbit in $I$, the induced subquiver has no arrows (that is, it is a discrete set).
\end{itemize}
\end{definition}

Let $Q = (I, H, s, t)$ be a finite quiver without oriented cycles, where $a$ is an admissible automorphism. Let $n$ be a fixed positive integer such that $a^n = 1$ on both $I$ and $H$.

Let $\mathbb{N}I$ denote the monoid of $\mathbb{N}$-linear combinations of vertices $i \in I$. Then the cyclic group $\langle a \rangle$ acts on $\mathbb{N}I$ via
\[
a \cdot \sum_{i \in I} \nu_i i = \sum_{i \in I} \nu_{i} a(i).
\]
Let $\mathbb{N}I^a \subset \mathbb{N}I$ denote the submonoid of fixed points under this action, i.e., consisting of $\sum_{i \in I} \nu_i i$ such that $\nu_i = \nu_{a(i)}$ for all $i \in I$.

For any $\nu \in \mathbb{N}I^a$, we fix an $I$-graded $k$-vector space $\mathbf{V}^\nu=\oplus_{i\in I}\bfV_i$ (or simply $\mathbf{V}$ when no confusion can arise) with dimension vector $\nu$, together with a map $a_{\bfV}:\bfV\rightarrow \bfV$ such that $a_{\bfV}:\bfV_i\rightarrow \bfV_{a(i)}$ is an isomorphism for every $i\in I$.

We define the algebraic group and affine space
\[
\mathbf{G}_{\nu} = \prod_{i \in I} \mathrm{GL}_k(\mathbf{V}_i), \quad
\mathbf{E}_{\nu} = \bigoplus_{h \in H} \mathrm{Hom}_k(\mathbf{V}_{s(h)}, \mathbf{V}_{t(h)}),
\]
where $\mathbf{G}_{\nu}$ acts on $\mathbf{E}_{\nu}$ via
\[
g \cdot x = \bigl(g_{t(h)} x_h g_{s(h)}^{-1}\bigr)_{h \in H}
\]
for all $g \in \mathbf{G}_{\nu}$ and $x \in \mathbf{E}_{\nu}$.

The automorphism $a$ induces an automorphism $a : \mathbf{G}_{\nu} \to \mathbf{G}_{\nu}$ satisfying
\[
\text{for all } g \in \mathbf{G}_{\nu}, \quad a(g) \circ a_{\mathbf{V}} = a_{\mathbf{V}} \circ g : \mathbf{V} \to \mathbf{V},
\]
and an automorphism $a : \mathbf{E}_{\nu} \to \mathbf{E}_{\nu}$ satisfying
\[
\text{for all } x \in \mathbf{E}_{\nu}, \ h \in H, \quad (a(x))_{h} \circ a_{\mathbf{V}} = a_{\mathbf{V}} \circ x_{a^{-1}(h)} : \mathbf{V}_{s(a^{-1}(h))} \to \mathbf{V}_{t(h)}.
\]
Let $\mathcal{S}_{\nu}$ denote the set of finite sequences $\boldsymbol{\nu} = (\nu^{1}, \nu^{2}, \ldots, \nu^{s})$, where each $\nu^{l} = a_{l} i_{l}$ with $a_{l} \in \mathbb{N}_{\geqslant 1}$ and $i_{l} \in I$. If
\[
\sum_{1 \leqslant l \leqslant s} \nu^{l} = \nu, \nu\in \bbN I,
\]
we call $\boldsymbol{\nu}$ a flag type of $\nu$, or equivalently of $\mathbf{V}$, and denote by $a(\boldsymbol{\nu})$ the sequence $(a(\nu^{1}), a(\nu^{2}), \ldots, a(\nu^{s}))$.

For a flag type $\boldsymbol{\nu}$ of $\mathbf{V}$, the corresponding flag variety $\mathcal{F}_{\boldsymbol{\nu},\Omega}$ is a smooth variety consisting of pairs $(x, f)$, where $x \in \mathbf{E}_{\mathbf{V},\Omega}$ and $f = ( 0= \mathbf{V}^{s} \subseteq \mathbf{V}^{s-1} \subseteq \cdots \subseteq \mathbf{V}^{0} = \mathbf{V})$ is an $I$-graded filtration such that $x(\mathbf{V}^{l}) \subseteq \mathbf{V}^{l}$ and
$
\dim \mathbf{V}^{l-1} / \mathbf{V}^{l} = \nu^{l}
$
for all $l$. In this case, there is a proper map
\[
\pi_{\boldsymbol{\nu}} : \mathcal{F}_{\boldsymbol{\nu},\Omega} \to \mathbf{E}_{\mathbf{V},\Omega}, \quad (x, f) \mapsto x.
\]
Hence, by the decomposition theorem \cite[Theorem 5.4.5]{BBD}, the complex
\[
L_{\boldsymbol{\nu}} = (\pi_{\boldsymbol{\nu},\Omega})_{!} \overline{\mathbb{Q}}_{l}|_{\mathcal{F}_{\boldsymbol{\nu},\Omega}} [\dim \mathcal{F}_{\boldsymbol{\nu},\Omega}]
\]
is a semisimple complex on $\mathbf{E}_{\mathbf{V},\Omega}$, where $\overline{\mathbb{Q}}_{l}|_{\mathcal{F}_{\boldsymbol{\nu},\Omega}}$ denotes the constant sheaf on $\mathcal{F}_{\boldsymbol{\nu},\Omega}$.

\begin{definition}
Let $\mathcal{P}_{\mathbf{V}}$ denote the set of simple perverse sheaves $L$ in $\mathcal{D}^{b}_{G_{\mathbf{V}}}(\mathbf{E}_{\mathbf{V},\Omega})$ satisfying the following condition: $L$ is a direct summand, up to shift, of some $L_{\boldsymbol{\nu}}$ associated with a flag type $\boldsymbol{\nu}$ of $\mathbf{V}$. Let $\mathcal{Q}_{\mathbf{V}}$ denote the full subcategory of $\mathcal{D}^{b}_{G_{\mathbf{V}}}(\mathbf{E}_{\mathbf{V},\Omega})$ consisting of direct sums, with shifts, of simple perverse sheaves in $\mathcal{P}_{\mathbf{V}}$. Following the notation of~\cite{OSNotes}, we call $\mathcal{Q}_{\mathbf{V}}$ Lusztig's category of sheaves associated with the quiver $Q$ and dimension vector $\nu$.
\end{definition}
The admissible automorphism $a$ induces a periodic functor $a^*$ on $\mq_{\bfV}$, because $a^*L_{a(\boldsymbol{\nu})}\cong L_{\boldsymbol{\nu}}$. We now introduce the periodic category $\widetilde{\mq_{\bfV}}$ associated with this functor.
\begin{definition}
   $\widetilde{\mq_{\bfV}}$ is a category whose objects are pairs $(B, \phi)$, where $B \in \mq_{\bfV}$ and 
\[
\phi : a^* B \rightarrow B
\]
is a morphism satisfying
\[
\phi \circ a^* \phi \circ \cdots \circ (a^*)^{n-1} \phi = \operatorname{id}.
\]
A morphism $f : (B, \phi) \to (B', \phi')$ in $\widetilde{\mq_{\bfV}}$ is a morphism $f : B \to B'$ in $\mq_{\bfV}$ satisfying
\[
f \circ \phi = \phi' \circ a^*(f).
\]
By Chapter~12 of \cite{lusztig2010introduction}, for the flag complexes $L_{\boldsymbol{\nu}}$, the following diagram is cartesian:
\[
\begin{tikzcd}
\mathcal{F}_{\boldsymbol{\nu},\Omega} \arrow[r, "a"] \arrow[d, "\pi_{\boldsymbol{\nu}}"] & \mathcal{F}_{a\boldsymbol{\nu},\Omega} \arrow[d, "\pi_{a\boldsymbol{\nu}}"] \\
\mathbf{E}_{\nu,\Omega} \arrow[r, "a"] & \mathbf{E}_{a\nu,\Omega}
\end{tikzcd}
\]
In this situation, we have
\[
a^*(\pi_{a\boldsymbol{\nu}})_{!}\overline{\mathbb{Q}}_{l}|_{\mathcal{F}_{a\boldsymbol{\nu},\Omega}}
\cong (\pi_{\boldsymbol{\nu}})_{!}a^*\overline{\mathbb{Q}}_{l}|_{\mathcal{F}_{\boldsymbol{\nu},\Omega}}
\cong (\pi_{\boldsymbol{\nu}})_{!}\overline{\mathbb{Q}}_{l}|_{\mathcal{F}_{\boldsymbol{\nu},\Omega}}.
\]
This yields a map
\[
a^* L_{a\boldsymbol{\nu}} \longrightarrow L_{\boldsymbol{\nu}},
\]
which we denote by $\phi_0$. If $a\boldsymbol{\nu}=\boldsymbol{\nu}$, then $(L_{\boldsymbol{\nu}},\phi_0)$ is fixed under Verdier duality $\mathbf{D}$, as follows from \cite[Lemma 12.3.2, Lemma 12.4.3]{lusztig2010introduction}.
\end{definition}
We denote by $\widetilde{P_{\bfV}}$ the full subcategory of $\widetilde{\mq_{\bfV}}$ consisting of objects $(B,\phi)$ with $B\in \mP_{\bfV}$.
\section{The folding extension algebra from $\widetilde{\mq_{\bfV}}$}
In this section, we consider the Ext-algebra arising from all Lusztig sheaves on $\mathbf{E}_{\mathbf{V}}$, where $\dim \mathbf{V} = \nu,$ $a(\nu)=\nu.$ We let $\zeta$ be a primitive $n$-th root of unity. Set $\mo=\bbZ[\zeta]$.
\subsection{The constructions of folding extension algebras}

Set
\[
L = \bigoplus_{\boldsymbol{\nu}\in \mathcal{S}} L_{\boldsymbol{\nu}}.
\]
The map $\phi_0 : a^* L_{a\boldsymbol{\nu}} \to L_{\boldsymbol{\nu}}$ naturally induces a map
\[
\phi_0 : a^* L \to L.
\]
Denote by
\[
R(\nu) = \mathrm{Ext}^{\bullet}_{\mathcal{D}^{b}_{\mathbf{G}_{\mathbf{V}}}(\mathbf{E}_{\mathbf{V}})}(L, L)
\]
the corresponding graded algebra. The map $\phi_0$ induces an algebra automorphism of $R(\nu)$ via the following diagram:
\[
\begin{tikzcd}
L \arrow[r, "f"] & L[n] \\
a^* L \arrow[u, "\phi_0"] & a^* L[n] \arrow[u, "\phi_0\text{[n]}"] \\
f \arrow[r, maps to] & (a^*)^{-1}(\phi_0[n]^{-1} \circ f \circ \phi_0)
\end{tikzcd}
\]
We denote this automorphism by
\[
\mathbf{a} : R(\nu) \to R(\nu).
\]
Since $(a^*)^{-n}(\phi_0) \circ \cdots \circ {a^*}^{-1}(\phi_0) = \operatorname{id}$, it follows that $\mathbf{a}^n = \operatorname{id}$.

Following McNamara's construction in \cite{https://doi.org/10.1112/jlms.12218}, for each $R(\nu)$-module $M$, one defines
$
\mathbf{a}^* M
$
via the composition $\rho_M \circ \mathbf{a}$, where $\rho_M : R(\nu) \to \mathrm{End}_{\overline{\bbq_l}}(M)$ denotes the representation. $\mathbf{a}^*$ is a functor between category $R(\nu)\text{-}\mathrm{mod}$ and itself. 
\begin{definition}
    The category $R(\nu)\text{-}\mathrm{mod}^\mathbf{a}$ is defined as follows: its objects are pairs $(M, \psi)$, where
\[
\psi_M : \mathbf{a}^* M \to M
\]
satisfies
\[
 \psi \circ \mathbf{a}^* \psi \circ \cdots \circ (\mathbf{a}^*)^{n-1} \psi= \operatorname{id}.
\]
Since $\psi$ is an element of $\mathrm{End}_{\overline{\bbq_l}}(M)$ and $\mathbf{a}^* \psi = \psi$ as a linear map, this condition reduces to
\[
\psi^n = \operatorname{id}.
\]
A morphism $f \in \mathrm{Hom}_{R(\nu)}(M, N)$ between two objects $(M, \psi)$ and $(N, \psi')$ in $R(\nu)\text{-}\mathrm{mod}^\mathbf{a}$ is required to satisfy
\[
f \circ \psi = \psi' \circ \mathbf{a}^* f.
\]
Equivalently, for all $r \in R(\nu)$ and $m \in M$, we have
\[
\psi(\rho_M(\mathbf{a}(r)) m) = \rho_M(r) \psi(m),
\]
which in $\mathrm{End}_{\overline{\bbq_l}}(M)$ reads
\[
\psi \circ \rho_M(\mathbf{a}(r)) = \rho_M(r) \circ \psi.
\]
\end{definition}
$ \langle \mathbf{a} \rangle= \langle \mathbf{a}^{-1} \rangle$ is a cyclic group of order $n.$ The skew group algebra $R(\nu)\langle\mathbf{a}\rangle$ of $R(\nu)$ with respect to $ \langle \mathbf{a} \rangle$ is defined following \cite{REITEN1985224}.
\begin{definition}
The skew group algebra $R(\nu)\langle\mathbf{a}\rangle$ is defined as the quotient of the free algebra $\langle R(\nu), \mathbf{a} \rangle$ by the relations
\[
\mathbf{a} r = \mathbf{a}(r) \mathbf{a}
\]
for all $r \in R(\nu)$. That is,
\[
R(\nu)\langle\mathbf{a}\rangle = \langle R(\nu), \mathbf{a} \rangle \big/ \langle \mathbf{a} r - \mathbf{a}(r) \mathbf{a} \mid r \in R(\nu) \rangle.
\]
\end{definition}
We now show that the category $R(\nu)\text{-}\mathrm{mod}^\mathbf{a}$ is equivalent to the category of $R(\nu)\langle\mathbf{a}\rangle$-modules.

\begin{proposition}
The category $R(\nu)\text{-}\mathrm{mod}^\mathbf{a}$ is equivalent to the category of $R(\nu)\langle\mathbf{a}\rangle$-modules.
\end{proposition}

\begin{proof}
Given an object $(M,\psi)$ in $R(\nu)\text{-}\mathrm{mod}^\mathbf{a}$, recall that $\mathbf{a}^*M$ corresponds to $\rho_M \circ \mathbf{a}$. Define a map
\[
\Theta : \langle R(\nu), \mathbf{a} \rangle \to \mathrm{End}_{\overline{\bbq_l}}(M)
\]
by setting $\Theta(r) = \rho_M(r)$ and $\Theta(\mathbf{a}) = \psi^{-1}$. Since $\psi^n = \mathrm{id}$, this is well-defined. Moreover, we have
\[
\psi \circ \rho_M(\mathbf{a}(r)) = \rho_M(r) \circ \psi,
\]
which translates to
\[
\Theta(\mathbf{a}^{-1} \cdot \mathbf{a}(r)) = \Theta(r \cdot \mathbf{a}^{-1}).
\]
Hence, $\Theta$ factors through $R(\nu)\langle\mathbf{a}\rangle$, giving $M$ the structure of an $R(\nu)\langle\mathbf{a}\rangle$-module. For a morphism $f : M \to N$ in $R(\nu)\text{-}\mathrm{mod}^\mathbf{a}$ satisfying
\[
f \circ \psi = \psi' \circ \mathbf{a}^* f,
\]
we have $f(\mathbf{a}^{-r} m) = \mathbf{a}^{-r} f(m)$ for all $\mathbf{a}^{-r} \in \langle\mathbf{a} \rangle$, $m \in M$, and hence $f$ naturally induces an $R(\nu)\langle\mathbf{a}\rangle$-module homomorphism.

Conversely, given an $R(\nu)\langle\mathbf{a}\rangle$-module $M$, restricting the action yields an $R(\nu)$-module. Let $\psi$ denote the action of $\mathbf{a}^{-1}$ on $M$. Since
\[
\mathbf{a}^{-1} \mathbf{a}(r) = r \mathbf{a}^{-1},
\]
it follows that
\[
\psi \circ \rho_M(\mathbf{a}(r)) = \rho_M(r) \circ \psi.
\]
A morphism $f$ between two $R(\nu)\langle\mathbf{a}\rangle$-modules $M$ and $N$ clearly satisfies the condition for a morphism between $(M,\psi)$ and $(N,\psi')$ (where $\psi'$ is the action of $\mathbf{a}^{-1}$ on $N$).

Hence, the categories are equivalent.
\end{proof}
\subsection{Projective modules in the folding extension algebra}
The object
\[
R(\nu) \xrightarrow{\mathbf{a}^{-1}} R(\nu)
\]
belongs to $R(\nu)\text{-}\mathrm{mod}^\mathbf{a}$. Under the equivalence above, it is a projective object in the category of $R(\nu)\langle\mathbf{a}\rangle$-modules.

\begin{lemma}\label{pro}
The objects
\[
R(\nu) \xrightarrow{\zeta^i\mathbf{a}^{-1}} R(\nu),\ i=0,\cdots,n-1,
\]
are projective in the category $R(\nu)\langle\mathbf{a}\rangle\text{-}\mathrm{mod}$ under the above equivalence. Moreover, their direct sum is $R(\nu)\langle\mathbf{a}\rangle$.
\end{lemma}

\begin{proof}
The object $R(\nu) \xrightarrow{\mathbf{a}^{-1}} R(\nu)$, viewed as an $R(\nu)\langle\mathbf{a}\rangle$-module, corresponds to a direct summand of $R(\nu)\langle\mathbf{a}\rangle$.

We have the decomposition
\[
R(\nu)\langle\mathbf{a}\rangle = \bigoplus_{i=0}^{n-1} (R(\nu)\langle\mathbf{a}\rangle)_i,
\]
where
\[
(R(\nu)\langle\mathbf{a}\rangle)_i := \left\{ \sum_{j=0}^{n-1} \zeta^{ij} r \mathbf{a}^j \mid r \in R(\nu) \right\}.
\]
The decomposition is evident, and it remains to verify that each $(R(\nu)\langle\mathbf{a}\rangle)_i$ is stable under the action of $a$. Indeed,
\[
\mathbf{a} \left( \sum_{j=0}^{n-1} \zeta^{ij} r \mathbf{a}^j \right)
= \sum_{j=0}^{n-1} \zeta^{ij} \mathbf{a}(r) \mathbf{a}^{j+1}
= \sum_{j=0}^{n-1} \zeta^{ij} \mathbf{a}(\zeta^{-i} r) \mathbf{a}^j,
\]
so stability holds.

It remains to check that
\[
R(\nu) \xrightarrow{\zeta^{i} \mathbf{a}^{-1}} R(\nu)
\]
corresponds to the $R(\nu)\langle\mathbf{a}\rangle$-module $(R(\nu)\langle\mathbf{a}\rangle)_i$. This follows from the map
\[
f_i : R(\nu) \to (R(\nu)\langle\mathbf{a}\rangle)_i, \quad f_i(r) = \sum_{j=0}^{n-1} \zeta^{ij} r \mathbf{a}^j,
\]
which gives an $R(\nu)\langle\mathbf{a}\rangle$-module isomorphism assuming $a$ acts as $\zeta^{-i}a$ on $R(\nu).$

The lemma follows.
\end{proof}
Given a complex of sheaves $\mathfrak{F}$, suppose there exists a map
\[
a^* \mathfrak{F} \xrightarrow{\phi} \mathfrak{F}
\]
such that
\[
\phi \circ a^* \phi \circ \cdots \circ (a^*)^{n-1} \phi = \mathrm{id}.
\]
Then consider $\mathrm{Ext}^{\bullet}(\mathfrak{F}, L)$, equipped with the map $\hat{\phi}$ induced by $\phi$ and $\phi_0$, defined by
\[
\hat{\phi}(f) := \phi_0[n] \circ a^*(f) \circ \phi^{-1},\ f\in \mathrm{Ext}^{n}(\mathfrak{F}, L).
\]
Notice that
\[
\phi_0[n+m] \circ a^*\bigl( (a^*)^{-1}(\phi_0[n+m]^{-1} \circ g \circ \phi_0[n]) \circ f \bigr) \circ \phi^{-1}
= g \circ \phi_0[n] \circ a^*(f) \circ \phi^{-1},\quad g\in \mathrm{Ext}^{m}(\mathfrak{F}, L),
\]
which shows that $\hat{\phi}$ defines an isomorphism
\[
\hat{\phi} : \mathbf{a}^* \mathrm{Ext}^{\bullet}(\mathfrak{F}, L) \to \mathrm{Ext}^{\bullet}(\mathfrak{F}, L).
\]

Furthermore, we may decompose
\[
R(\nu) \xrightarrow{\mathbf{a}^{-1}} R(\nu)
\]
as the direct sum of
\[
(\mathbf{a}^* \mathrm{Ext}^{\bullet}(P, L) \xrightarrow{\hat{\phi}_P} \mathrm{Ext}^{\bullet}(P, L))
\]
and
\[
\Bigl( \mathbf{a}^* \mathrm{Ext}^{\bullet}\Bigl( \bigoplus_{k=0}^{l} (a^*)^k P, L \Bigr) 
\xrightarrow{\hat{\phi}_{\oplus}} 
\mathrm{Ext}^{\bullet}\Bigl( \bigoplus_{k=0}^{l} (a^*)^k P, L \Bigr) \Bigr),
\]
where $\phi_P : a^* P \to P$ is the $a^*$-invariant summand in $\widetilde{\mathcal{P}_{\mathbf{V}}}$ defined by Lusztig (see Proposition 12.6.3 of \cite{lusztig2010introduction}), and $\phi_{\oplus} : a^*(\bigoplus_{k=0}^{l} (a^*)^k P) \to \bigoplus_{k=0}^{l} (a^*)^k P$ is a traceless component.

\begin{corollary}\label{apro}
All indecomposable projective modules over $R(\nu)\langle\mathbf{a}\rangle$ are of the form
\[
(\mathbf{a}^* \mathrm{Ext}^{\bullet}(P, L) \xrightarrow{\zeta^i \hat{\phi}_P} \mathrm{Ext}^{\bullet}(P, L)),
\]
where $a^* P \cong P$ and $P$ is a simple perverse sheaf as a direct summand of $L$ up to shift, or
\[
\Bigl( \mathbf{a}^* \mathrm{Ext}^{\bullet}\Bigl( \bigoplus_{k=0}^{l} (a^*)^k P, L \Bigr) 
\xrightarrow{\hat{\phi}_{\oplus}} 
\mathrm{Ext}^{\bullet}\Bigl( \bigoplus_{k=0}^{l} (a^*)^k P, L \Bigr) \Bigr),
\]
where $a^* P \ncong P$, $l+1$ is the minimal positive integer such that $(a^*)^{l+1} P \cong P$, and $P$ is a simple perverse sheaf as a direct summand of $L$ up to shift.
\end{corollary}

\begin{proof}
Applying Lemma \ref{pro}, it remains only to show that these modules are indecomposable. Since $\mathrm{Ext}^{\bullet}(P, L)$ is an indecomposable $R(\nu)$-module by \cite[Theorem 4.4]{VaragnoloVasserot+2011+67+100}, it follows that
\[
(a^* \mathrm{Ext}^{\bullet}(P, L) \xrightarrow{\zeta^i \hat{\phi}_P} \mathrm{Ext}^{\bullet}(P, L))
\]
is also an indecomposable representation of $R(\nu)\langle\mathbf{a}\rangle$.

For
\[
\Bigl( a^* \mathrm{Ext}^{\bullet}\Bigl( \bigoplus_{k=0}^{l} (a^*)^k P, L \Bigr) 
\xrightarrow{\hat{\phi}_{\oplus}} 
\mathrm{Ext}^{\bullet}\Bigl( \bigoplus_{k=0}^{l} (a^*)^k P, L \Bigr) \Bigr)
\]
where $a^* P \ncong P$ and $P$ is a simple perverse sheaf, quotienting by the radical gives
\[
\mathbf{a}^* \Bigl( \bigoplus_{k=0}^{l} (\mathbf{a}^*)^k S \Bigr) \to \bigoplus_{k=0}^{l} (\mathbf{a}^*)^k S,
\]
where $S$ is a simple $R(\nu)$-module, and $\mathbf{a}^* S \ncong S$. Since $l+1$ is minimal with $(\mathbf{a}^*)^{l+1} S \cong S$, this is a simple $R(\nu)\langle\mathbf{a}\rangle$-module. The result follows.
\end{proof}
For Lusztig's induction functor,
\[
\Ind_{\nu',\nu''}^{\nu}:
\cD_{\bfG_{\bfV'}}^b(\bfE_{\bfV'})\times
\cD_{\bfG_{\bfV''}}^b(\bfE_{\bfV''})
\rightarrow \cD_{\bfG_{\bfV}}^b(\bfE_{\bfV}),
\]
where the dimension vectors of the $I$-graded vector spaces $\bfV, \bfV', \bfV''$ are $\nu, \nu', \nu''$ stable under $a$, respectively, define
\[
\bfE_{\bfV}'=\{(x,\bfW,\rho_1,\rho_2)\mid x\in \bfE_{\bfV},\ \bfW\cong \bfV''\text{ as an }I\text{-graded vector space},\
\rho_1:\bfV/\bfW\cong \bfV',\ \rho_2:\bfW\cong \bfV''\}
\]
and
\[
\bfE_{\bfV}''=\{(x,\bfW)\mid x\in \bfE_{\bfV},\ \bfW\cong \bfV''\text{ as an }I\text{-graded vector space}\}.
\]
The maps are defined by
\[
p_1(x,\bfW,\rho_1,\rho_2)=(\rho_1(x|_{\bfV/\bfW})\rho_1^{-1},\rho_2(x|_{\bfW})\rho_2^{-1}),
\quad p_2(x,\bfW,\rho_1,\rho_2)=(x,\bfW),
\quad p_3(x,\bfW)=x.
\]
$p_1$ is smooth of affine bundles and $p_2$ is $\bfG_{\bfV'}\times \bfG_{\bfV''}$ principle.
Thus we have the diagram
\[
\begin{tikzcd}
\bfE_{\bfV'}\times \bfE_{\bfV''} & \bfE_{\bfV}' \arrow[l, "p_1"'] \arrow[r, "p_2"] & \bfE_{\bfV}'' \arrow[r, "p_3"] & \bfE_{\bfV}.
\end{tikzcd}
\]
We denote the dimensions of the fibers of the smooth morphisms $p_1$ and $p_2$ by $d_1$ and $d_2$, respectively. The induction functor is defined by
\[
\Ind_{\nu',\nu''}^{\nu} := {p_3}_! {p_2}_b p_1^* [d_1 - d_2] \left( \frac{d_1 - d_2}{2} \right),
\] where ${p_2}_b$ is the inverse functor of $p_2^*.$
Since $a^* {p_3}_! \cong {p_3}_! a^*$, there is a natural transformation
\[
a^* \Ind_{\nu',\nu''}^{\nu} \rightarrow \Ind_{\nu',\nu''}^{\nu} a^*.
\]
For a fixed subspace $\bfW\subset \bfV$, $\dim \bfW=\dim \bfV''$, $\bfV\cong \bfV'\oplus\bfV''$, $P$ is the stabilizer of $\bfW$ in $G_{\bfV}$ and $F=\{x\in \bfE_{\bfV}|x(\bfW)\subset \bfW\},$ $$\bfE_{\bfV_1}\times \bfE_{\bfV_2}\xleftarrow{\kappa}F\xrightarrow{\iota}\bfE_{\bfV},$$

where $\iota$ is the embedding and $\kappa(x)=(x|_{\bfV/\bfW},x|_{\bfW})$.

Lusztig has defined the restriction functor $\mathbf{Res}_{\nu',\nu''}^{\nu}=(\kappa)_!(\iota)^*[d_1-d_2-2\dim G_{\bfV}/P].$
there is a natural transformation
\[
a^* \Res_{\nu',\nu''}^{\nu} \rightarrow \Res_{\nu',\nu''}^{\nu} a^*.
\]
We use the K\"unneth isomorphism
\[
\mathcal{H}om_{\cD^b_{\bfG_{\bfV'}\times \bfG_{\bfV''}}(\bfE_{\bfV'}\times \bfE_{\bfV''})}
(L_{\boldsymbol{\nu'}}\boxtimes L_{\boldsymbol{\nu''}},L_{\boldsymbol{\nu_1'}}\boxtimes L_{\boldsymbol{\nu_1''}})
\cong
\mathcal{H}om_{\cD^b_{\bfG_{\bfV'}}(\bfE_{\bfV'})}(L_{\boldsymbol{\nu'}},L_{\boldsymbol{\nu_1'}})
\otimes
\mathcal{H}om_{\cD^b_{\bfG_{\bfV''}}(\bfE_{\bfV''})}(L_{\boldsymbol{\nu''}},L_{\boldsymbol{\nu_1''}}).
\]
If
\[
L'=\bigoplus_{\boldsymbol{\nu}=\boldsymbol{\nu'}\boldsymbol{\nu''},\ \boldsymbol{\nu'}\in \mathcal{S}_{\nu'},\ \boldsymbol{\nu''}\in \mathcal{S}_{\nu''}} L_{\boldsymbol{\nu}}
\quad \text{and} \quad
R(\nu)_{\nu',\nu''}=\Ext^{\bullet}(L',L'),
\]
then, by \cite[4.1]{VaragnoloVasserot+2011+67+100}, there is an embedding
$R(\nu')\otimes R(\nu'')\hookrightarrow R(\nu)_{\nu',\nu''}$, and by \cite[Theorem 5.7 and 5.4.3]{doi:10.1142/S1005386712000247} this embedding is induced by $\Ind_{\nu',\nu''}^{\nu}$. Hence
\begin{equation}\label{1}
\Ind_{\nu',\nu''}^{\nu}:
\Ext^{\bullet}(L_{\boldsymbol{\nu'}}\boxtimes L_{\boldsymbol{\nu''}},L_{\boldsymbol{\nu_1'}}\boxtimes L_{\boldsymbol{\nu_1''}})
\rightarrow
\Ext^{\bullet}(L_{\boldsymbol{\nu}},L_{\boldsymbol{\nu_1}})
\end{equation}
is injective. And this injective map is compatible with $R(\nu)\langle\mathbf{a}\rangle-$module structure induced by $\hat{\phi_0}.$

The induction functor
$\ind_{\nu',\nu''}^{\nu}: \mod R(\nu')\times \mod R(\nu'')\rightarrow \mod R(\nu)$
defined in \cite[4.1]{VaragnoloVasserot+2011+67+100} sends $M\otimes N$ to
$R(\nu)\otimes_{R(\nu')\otimes R(\nu'')}(M\otimes N)$. Restricted to the projective subcategories, the injectivity in~\eqref{1} shows that $\ind_{\nu',\nu''}^{\nu}$ is induced by 
\[
\Ind_{\nu',\nu''}^{\nu}:R(\nu')\otimes R(\nu'')\rightarrow R(\nu)_{\nu',\nu''}
\] on projective modules.

The morphism
\[
\mathbf{a}^*(R(\nu)\otimes_{R(\nu')\otimes R(\nu'')}(M\otimes N))
\rightarrow R(\nu)\otimes_{R(\nu')\otimes R(\nu'')}(M\otimes N)
\]
is induced by $R(\nu)\xrightarrow{\mathbf{a}^{-1}}R(\nu)$ and by the map $\mathbf{a}^*(M\otimes N)\rightarrow M\otimes N$.

By \cite[Lemma 12.3.2]{lusztig2010introduction},
\begin{equation}\label{eq11}
    \Ind_{\nu',\nu''}^{\nu}((L_{\boldsymbol{\nu'}},\phi_0)\boxtimes (L_{\boldsymbol{\nu''}},\phi_0))=(L_{\boldsymbol{\nu}},\phi_0),
\end{equation}
gives the morphism induced by $(\Ext^{\bullet}(L_{\boldsymbol{\nu'}}\boxtimes L_{\boldsymbol{\nu''}},L_{\boldsymbol{\nu_1'}}\boxtimes L_{\boldsymbol{\nu_1''}}),\widehat{\phi_0\boxtimes\phi_0})$ and $(\Ext^{\bullet}(L,L),\hat{\phi_0}).$

The restriction functor acting on projective $R(\nu)$ modules defined in \cite{VaragnoloVasserot+2011+67+100} is $$\res_{\nu',\nu''}^{\nu}(\Ext^{\bullet}(P,L))=\Ext^{\bullet}(P,\Ind_{\nu',\nu''}^{\nu}(L'\boxtimes L''))=\Ext^{\bullet}(\Res_{\nu',\nu''}^{\nu}(P),L'\boxtimes L'')$$ as in \cite[Definition 7.4]{Bi2024}, where $L'$ and $L''$ are the direct sum of flag sheaves complexes on $\bfE_{\bfV"}$ and $\bfE_{\bfV''}.$
It induce a restriction functor on projective $R(\nu)\langle \mathbf{a}\rangle$ modules.

We denote by $\mathcal{P}_{\nu}$ the full subcategory of $R(\nu)\langle\mathbf{a}\rangle$-modules consisting of modules of finite projective dimension, and by $\mathcal{L}_{\nu}$ the full subcategory consisting of modules finitely generated by simple modules.

Next, following the notation in \cite{https://doi.org/10.1112/jlms.12218}, we consider two kinds of dualities, both denoted by $\mathbb{D}$. For an object $(M, \phi)$ in $\mathcal{P}_{\nu}$, its dual is defined as
\[
\mathbb{D}M = \mathrm{Hom}_{R(\nu)\text{-}\mathrm{mod}}(M, R(\nu))
\]
(without requiring grading), equipped with the natural $R(\nu)$-module structure: for $\lambda\in \mathrm{Hom}_{R(\nu)\text{-}\mathrm{mod}}(M, R(\nu))$ and $r\in R(\nu)$, define
\[
(r\cdot \lambda)(m)=\lambda(rm),
\]
where the $R(\nu)$-action is obtained through the contravariant isomorphism.

The map $\phi$ induces a map
\[
(\mathbb{D}\phi)^{-1} : \mathbf{a}^* \mathrm{Hom}_{R(\nu)\text{-}\mathrm{mod}}(M, R(\nu)) \to \mathrm{Hom}_{R(\nu)\text{-}\mathrm{mod}}(M, R(\nu))
\]
given by
\[
(\mathbb{D}\phi)^{-1}(f) = f \circ \phi^{-1}.
\]

For an object $(K, \psi)$ in $\mathcal{L}_{\nu}$, define its dual as
\[
\mathbb{D}K = \mathrm{Hom}_{\overline{\bbq_l}}(K, {\overline{\bbq_l}}).
\]
The map $\psi$ induces a map
\[
(\mathbb{D}\psi)^{-1} : \mathbf{a}^* \mathrm{Hom}_{\overline{\bbq_l}}(K, {\overline{\bbq_l}}) \to \mathrm{Hom}_{\overline{\bbq_l}}(K, {\overline{\bbq_l}})
\]
which sends $f \in \mathrm{Hom}_{\overline{\bbq_l}}(K, {\overline{\bbq_l}})$ to $f \circ \psi^{-1}$.

We now present a class of projective $R(\nu)\langle\mathbf{a}\rangle$-modules that are invariant under duality.

\begin{proposition}\label{dual}
Let $(P, \phi_P)$ be an object in $\widetilde{\mathcal{P}_{\mathbf{V}}}$ with $P$ a simple perverse sheaf, and suppose that $\mathbf{D}(P, \phi_P) \cong (P, \phi_P)$ under Verdier duality. Then $(\mathrm{Ext}^{\bullet}(P, L), \hat{\phi_P})$ is self-dual under the duality $\mathbb{D}$.
\end{proposition}

\begin{proof}
Indeed,
\[
\mathrm{Hom}_{R(\nu)\text{-}\mathrm{mod}}(\mathrm{Ext}^{\bullet}(P, L), \mathrm{Ext}^{\bullet}(L, L)) = \mathrm{Ext}^{\bullet}(L, P),
\]
and the induced map is given by the diagram:
\[
\begin{tikzcd}
a^* L \arrow[d, "\phi_0"] & a^* P \arrow[d, "\phi_P\text{[n]}"] \\
L \arrow[r, "f"] & P \\
f \arrow[r, maps to] & \phi_P[n] \circ a^*(f) \circ \phi_0^{-1}.
\end{tikzcd}
\]
Since both $(P, \phi_P)$ and $(L, \phi_0)$ are preserved under Verdier duality $\mathbf{D}$, we may describe $\mathrm{Ext}^{\bullet}(L, P)$ by $\mathbf{D}((L, \phi_0) \otimes \mathbf{D}(P, \phi_P))$ and $\mathrm{Ext}^{\bullet}(P, L)$ by $\mathbf{D}((P, \phi_P) \otimes \mathbf{D}(L, \phi_0))$. The $R(\nu)$-module structure is the one given by the definition of the dual module, and the isomorphism between $\mathbf{D}((P, \phi_P) \otimes \mathbf{D}(L, \phi_0))$ and $\mathbf{D}((L, \phi_0) \otimes \mathbf{D}(P, \phi_P))$ identifies $\hat{\phi_P}$ with $\mathbb{D}(\hat{\phi_P})^{-1}$. It follows that $(\mathrm{Ext}^{\bullet}(P, L), \hat{\phi_P})$ is self-dual under $\mathbb{D}$.
\end{proof}
We define the Grothendieck groups of $\mathcal{P}_{\nu}$ and $\mathcal{L}_{\nu}$. From the proof of Lemma~\ref{pro}, it follows that if $(\mathbf{P}, \phi)$ is a projective $R(\nu)\langle\mathbf{a}\rangle$-module, then $(\mathbf{P}, \zeta \phi)$ is also a projective $R(\nu)\langle\mathbf{a}\rangle$-module. Similarly, if $(\mathbf{L}, \psi)$ is a simple $R(\nu)\langle\mathbf{a}\rangle$-module, then $(\mathbf{L}, \zeta \psi)$ is also simple. Therefore, if $(M, \phi) \in \mathcal{P}_{\nu}$, then $(M, \zeta \phi) \in \mathcal{P}_{\nu}$; and if $(N, \psi) \in \mathcal{L}_{\nu}$, then $(N, \zeta \psi) \in \mathcal{L}_{\nu}$.

\begin{definition}
The Grothendieck group of $\mathcal{P}_{\nu}$, denoted by $\mathcal{K}(\mathcal{P}_{\nu})$, is defined as the $\mathcal{O}[t, t^{-1}]$-module generated by isomorphism classes $[(M,\phi)]$ of objects $(M,\phi)\in \mathcal{P}_{\nu}$, subject to the following relations:
\begin{enumerate}
    \item If $(M,\phi)\langle n\rangle$ denotes the $n$-th grading shift of $(M,\phi)$, impose $[(M,\phi)\langle n\rangle] = t^{n}[(M,\phi)]$.
    \item For each graded short exact sequence
    \[
    0 \to (M,\phi_M) \to (N,\phi_N) \to (K,\phi_K) \to 0
    \]
    in $R(\nu)\langle\mathbf{a}\rangle$-modules, impose $[(M,\phi_M)] + [(K,\phi_K)] = [(N,\phi_N)]$.
    \item $[(M,\zeta \phi_M)] = \zeta [(M,\phi_M)]$.
    \item $[(M \oplus a^* M \oplus \cdots \oplus (a^*)^l M, \phi)] = 0$ if $\phi$ is a permutation morphism and $l > 0$. Such elements are called traceless elements.
\end{enumerate}
Similarly, the Grothendieck group of $\mathcal{L}_{\nu}$ is denoted by $\mathcal{K}(\mathcal{L}_{\nu})$.
\end{definition}

The direct sum $\oplus_{\nu}\mathcal{K}(\mathcal{P}_{\nu})$ carries an bialgebra structure for which $\ind$ gives multiplication and $\Res$ gives comultiplication.
We have the following theorem for any symmetrizable case.

\begin{theorem}
The algebra $\oplus_{\nu}\mathcal{K}(\mathcal{P}_{\nu})$ is bialgebraic isomorphic to $\oplus_{\bfV}\mk(\widetilde{\mq_{\bfV}})$ \cite[Chapter 12]{lusztig2010introduction}. Moreover,
\[
\ind_{\nu',\nu''}^{\nu}((\Ext^{\bullet}(L_{\boldsymbol{\nu'}},\oplus_{\boldsymbol{\omega'}\in \mathcal{S}_{\nu'}}L_{\boldsymbol{\omega'}}),\hat{\phi_0})\otimes \Ext^{\bullet}((L_{\boldsymbol{\nu''}},\oplus_{\boldsymbol{\omega''}\in \mathcal{S}_{\nu''}}L_{\boldsymbol{\omega''}},\hat{\phi_0})))\cong (\Ext^{\bullet}(L_{\boldsymbol{\nu'\nu''}},\oplus_{\boldsymbol{\omega}\in \mathcal{S}_{\nu}}L_{\boldsymbol{\omega}}),\hat{\phi_0})
\]
Consequently, this algebra is isomorphic to the negative part of the quantum group associated with $(Q,a)$, and $\Ext^{\bullet}(\ ,L)$ preserves the signed basis and the crystal structure.

Under this isomorphism, the classes $[(\mathrm{Ext}^{\bullet}(P, L), \hat{\phi_P})]$, where $(P,\phi_P)\in \widetilde{P_{\bfV}}$ and $\mathbf{D}(P,\phi_P)\cong(P,\phi_P)$, give the signed canonical basis. 
\end{theorem}
\begin{proof}
   Let $\boldsymbol{\nu}=\boldsymbol{\nu'\nu''}.$
   And let the morphism from  $\mk(\widetilde{\mq_{\bfV}})$ to $\mathcal{K}(\mathcal{P}_{\nu})$ map $[(B,\phi)]$ to $[(\Ext^{\bullet}(B,L),\hat{\phi})].$ It is well defined since it maps traceless elements to traceless elements. And considering the basis, this map is an isomorphism. Since $\oplus_{\bfV}\mk(\widetilde{\mq_{\bfV}})$ is spanned by $[L_{\boldsymbol{\nu}},\phi_0],$ we then prove the isomorphism is a bialgebraic isomorphism on $[L_{\boldsymbol{\nu}},\phi_0].$
   
   The isomorphism
    $\ind_{\nu',\nu''}^{\nu}(\Ext^{\bullet}(L_{\boldsymbol{\nu'}},\oplus_{\boldsymbol{\omega'}\in \mathcal{S}_{\nu'}}L_{\boldsymbol{\omega'}})\otimes \Ext^{\bullet}(L_{\boldsymbol{\nu''}},\oplus_{\boldsymbol{\omega''}\in \mathcal{S}_{\nu''}}L_{\boldsymbol{\omega''}}))\cong \Ext^{\bullet}(L_{\boldsymbol{\nu}},\oplus_{\boldsymbol{\omega}\in \mathcal{S}_{\nu}}L_{\boldsymbol{\omega}})$ is induced by $\Ind_{\nu',\nu''}^{\nu}$, as in \cite[Remark 4.8]{VaragnoloVasserot+2011+67+100}. Thus, by equation~\ref{eq11},
    $$\ind_{\nu',\nu''}^{\nu}((\Ext^{\bullet}(L_{\boldsymbol{\nu'}},\oplus_{\boldsymbol{\omega'}\in \mathcal{S}_{\nu'}}L_{\boldsymbol{\omega'}}),\hat{\phi_0})\otimes \Ext^{\bullet}((L_{\boldsymbol{\nu''}},\oplus_{\boldsymbol{\omega''}\in \mathcal{S}_{\nu''}}L_{\boldsymbol{\omega''}},\hat{\phi_0})))\cong (\Ext^{\bullet}(L_{\boldsymbol{\nu}},\oplus_{\boldsymbol{\omega}\in \mathcal{S}_{\nu}}L_{\boldsymbol{\omega}}),\hat{\phi_0}).$$

    By the definition of $\res_{\nu',\nu''}^{\nu},$ $$\res_{\nu',\nu''}^{\nu}(\Ext^{\bullet}(L_{\boldsymbol{\nu}},\oplus_{\boldsymbol{\omega}\in \mathcal{S}_{\nu}}L_{\boldsymbol{\omega}}))=\Ext^{\bullet}(\Res_{\nu',\nu''}^{\nu}(L_{\boldsymbol{\nu}}),(\oplus_{\boldsymbol{\omega'}\in \mathcal{S}_{\nu'}}L_{\boldsymbol{\omega'}})\boxtimes(\oplus_{\boldsymbol{\omega''}\in \mathcal{S}_{\nu''}}L_{\boldsymbol{\omega''}})).$$ As in \cite[Lemma 12.3.3]{lusztig2010introduction}, $\Res_{\nu',\nu''}^{\nu}(L_{\boldsymbol{\nu}},\phi_0)=\sum_{\boldsymbol{\nu'},\boldsymbol{\nu''}}(L_{\boldsymbol{\nu'}},\phi_0)\boxtimes (L_{\boldsymbol{\nu''}},\phi_0)[M(\boldsymbol{\nu'},\boldsymbol{\nu''})]\oplus T,$ where $T$ is a traceless element and $M(\boldsymbol{\nu'},\boldsymbol{\nu''})$ is a constant number depends on $\boldsymbol{\nu'},\boldsymbol{\nu''}.$ Then \begin{align*}
        [&\res_{\nu',\nu''}^{\nu}(\Ext^{\bullet}(L_{\boldsymbol{\nu}},\oplus_{\boldsymbol{\omega}\in \mathcal{S}_{\nu}}L_{\boldsymbol{\omega}}),\hat{\phi_0})]=\\&\sum_{\boldsymbol{\nu'},\boldsymbol{\nu''}}t^{M(\boldsymbol{\nu'},\boldsymbol{\nu''})}[(\Ext^{\bullet}(L_{\boldsymbol{\nu'}},\oplus_{\boldsymbol{\omega'}\in \mathcal{S}_{\nu'}}L_{\boldsymbol{\omega'}}),\hat{\phi_0})\otimes (\Ext^{\bullet}(L_{\boldsymbol{\nu''}},\oplus_{\boldsymbol{\omega''}\in \mathcal{S}_{\nu''}}L_{\boldsymbol{\omega''}}),\hat{\phi_0})].
    \end{align*}  Thus the morphism is a bialgebraic morphism.
    The assertion of basis now follows from Corollary~\ref{apro} and Proposition~\ref{dual}.
\end{proof}
This theorem gives the main result in \cite{https://doi.org/10.1112/jlms.12218}. 

We now define a bilinear pairing on $\mathcal{K}(\mathcal{P}_{\nu}) \times \mathcal{K}(\mathcal{L}_{\nu})$.

\begin{definition}
For $R(\nu)$-modules $M$ and $N$, define
\[
\dim \mathrm{Hom}_{R(\nu)\text{-}\mathrm{mod}}(M,N) = \sum_{i\in \mathbb{Z}} \dim \mathrm{Hom}_{R(\nu)\text{-}\mathrm{gmod}}(M,N)_i t^i,
\]
where $\mathrm{Hom}_{R(\nu)\text{-}\mathrm{gmod}}(M,N)_i$ denotes the degree-zero maps in $\mathrm{Hom}_{R(\nu)\text{-}\mathrm{mod}}(M,N\langle i\rangle)$.

Given $(M,\phi)$ and $(N,\psi)$, the maps $\phi$ and $\psi$ induce a map
\[
a_{\phi,\psi} : \mathrm{Hom}_{R(\nu)\text{-}\mathrm{mod}}(M,N)_i \to \mathrm{Hom}_{R(\nu)\text{-}\mathrm{mod}}(M,N)_i
\]
given by
\[
a_{\phi,\psi}(f) = (\mathbf{a}^*)^{-1}(\psi^{-1} \circ f \circ \phi).
\]
The pairing on $\mk(\mP_{\nu})\times \mk(\ml_{\nu})$ is defined as 
\[
\langle [(M,\phi)], [(N,\psi)] \rangle_{\nu} = \sum_{i,j \in \mathbb{Z}} (-1)^j \mathrm{tr}\bigl(a_{\phi,\psi[j]}, \mathrm{Hom}_{\mathcal{D}^{b}(R(\nu)\text{-}\mathrm{mod})}(M,N[j])_i\bigr) t^i.
\]
\end{definition}

The pairing is well-defined. Relations $(1)$, $(3)$, and $(4)$ are immediate. For relation $(2)$, any short exact sequence
\[
0 \to (M,\phi_M) \to (N,\phi_N) \to (K,\phi_K) \to 0
\]
induces a long exact sequence after applying $\mathrm{Hom}$; hence the corresponding relation in the Grothendieck group lies in the radical of the pairing. Moreover, this pairing is a $\mathbb{Z}$-bilinear form and satisfies
\[
\langle t^i \zeta^j [(M,\phi)], t^{i'} \zeta^{j'} [(N,\psi)] \rangle_{\nu} = t^{i - i'} \zeta^{j - j'} \langle [(M,\phi)], [(N,\psi)] \rangle_{\nu}.
\]

By Corollary~\ref{apro} and Proposition~\ref{dual}, it follows that all non-traceless, self-dual, indecomposable projective $R(\nu)\langle\mathbf{a}\rangle$-modules are of the form
\[
(\mathrm{Ext}^{\bullet}(P, L), \hat{\phi_P})
\]
where $(P,\phi_P)$ is a simple perverse sheaf that is invariant under Verdier duality. We denote such modules by $(\mathbf{P}, \hat{\phi_P})$.

Moreover, the top of every self-dual indecomposable projective module is a self-dual simple module, and it suffices to consider the non-traceless part. The simple module corresponding to $(\mathbf{P}_P,\hat{\phi_P})$ is denoted by $(\mathbf{L}_P,\phi_{L_P})$.

It is clear that the classes $[(\mathbf{P}_P,\hat{\phi_P})]$ form a basis of $\mathcal{K}(\mathcal{P}_{\nu})$, and the classes $[(\mathbf{L}_P,\phi_{L_P})]$ form a basis of $\mathcal{K}(\mathcal{L}_{\nu})$. Under the pairing, these two bases satisfy
\[
\langle [(\mathbf{P}_P,\hat{\phi_P})], [(\mathbf{L}_{P'},\phi_{L_{P'}})] \rangle_{\nu} = \delta_{P,P'}.
\]

\subsection{Dual Standard Modules and Standard Modules in the Symmetrizable Case}\label{dual standard}
In this subsection, we restrict our attention to the finite type case, where the Ext-algebra arises from a finite type quiver.

In this setting, Kato proved that $R(\nu)$ has finite global dimension~\cite{cd088214-009e-3813-9a86-97f9a299a24c}. Combining this with \cite[Lemma 3.1]{REITEN1985224}, we obtain the following proposition.
\begin{proposition}
The global dimension of $R(\nu)\langle\mathbf{a}\rangle$ is finite.
\end{proposition}

\begin{proof}
There is a natural embedding $R(\nu) \hookrightarrow R(\nu)\langle\mathbf{a}\rangle$. For an $R(\nu)$-module $M$, consider the map
\[
i_M : M \to R(\nu)\langle\mathbf{a}\rangle \otimes_{R(\nu)} M, \quad i_M(m) = 1 \otimes m, \quad \forall m \in M.
\]
Define
\[
p_M : R(\nu)\langle\mathbf{a}\rangle \otimes_{R(\nu)} M \to M
\]
by
\[
p_M(\mathbf{a}^k \otimes m) = 0 \quad \text{for } k \ne 0, \quad p_M(1 \otimes m) = m.
\]
Clearly, $p_M \circ i_M = \mathrm{id}_M$. For an $R(\nu)\langle\mathbf{a}\rangle$-module $N$, define
\[
q_N : R(\nu)\langle\mathbf{a}\rangle \otimes_{R(\nu)} N \to N
\]
by $q_N(r \otimes x) = r x$ for all $r \in R(\nu)\langle\mathbf{a}\rangle$ and $x \in N$.

Now define
\[
\tilde{i_N} : N \to R(\nu)\langle\mathbf{a}\rangle \otimes_{R(\nu)} N
\]
by
\[
\tilde{i_N}(x) = \frac{1}{n} \sum_{k=0}^{n-1} \mathbf{a}^{-k} i_N(\mathbf{a}^k x).
\]
One can check that $\tilde{i_N}$ is an $R(\nu)\langle\mathbf{a}\rangle$-module homomorphism, and satisfies $q_N \circ \tilde{i_N} = \mathrm{id}_N$.

Therefore, for any $R(\nu)$-module $M$, we have:
\[
\mathrm{proj.dim}_{R(\nu)} M \geq \mathrm{proj.dim}_{R(\nu)\langle\mathbf{a}\rangle}(R(\nu)\langle\mathbf{a}\rangle \otimes_{R(\nu)} M)
\geq \mathrm{proj.dim}_{R(\nu)}(R(\nu)\langle\mathbf{a}\rangle \otimes_{R(\nu)} M) \geq \mathrm{proj.dim}_{R(\nu)} M.
\]

For any $R(\nu)\langle\mathbf{a}\rangle$-module $N$, since $N$ is a direct summand of $R(\nu)\langle\mathbf{a}\rangle \otimes_{R(\nu)} N$, it follows that
\[
\mathrm{proj.dim}_{R(\nu)\langle\mathbf{a}\rangle} N \leq \mathrm{proj.dim}_{R(\nu)\langle\mathbf{a}\rangle}(R(\nu)\langle\mathbf{a}\rangle \otimes_{R(\nu)} N).
\]
Hence,
\[
\mathrm{proj.dim}(R(\nu)) = \mathrm{proj.dim}(R(\nu)\langle\mathbf{a}\rangle).
\]
The proof is complete.
\end{proof}

Next, we construct the dual standard module $\tilde{K}_{\lambda}$ and the map $\mathbf{a}^* \tilde{K}_{\lambda} \to \tilde{K}_{\lambda}$. We realize both through the Ext-algebra in the derived category of sheaves. Let $\Lambda$ denote the set of $\mathbf{G}_{\mathbf{V}}$-orbits in $\mathbf{E}_{\mathbf{V}}$, and let $\lambda \in \Lambda$. Let $\mathcal{O}_{\lambda}$ denote the orbit corresponding to $\lambda$. Define the partial order on $\Lambda$ by $\mu \leq \lambda$ if $\mathcal{O}_{\mu} \subset \overline{\mathcal{O}_{\lambda}}$. Let
\[
j_{\lambda} : \mathcal{O}_{\lambda} \hookrightarrow \mathbf{E}_{\mathbf{V},\Omega}
\]
denote the inclusion.

Since we are in the finite type case, the space $\mathbf{E}_{\mathbf{V},\Omega}$ is a linear space defined over $\overline{\mathbb{F}_{q}}$, and $(\mathbf{E}_{\mathbf{V},\Omega})_0$ is its $\bbF_q$-form, so that $\mathbf{E}_{\mathbf{V},\Omega}\cong \operatorname{spec}(\overline{\mathbb{F}_{q}})\times_{\operatorname{spec}(\bbF_q)}(\mathbf{E}_{\mathbf{V},\Omega})_0$. Let $\gamma:\operatorname{spec}(\overline{\mathbb{F}_{q}})\times_{\operatorname{spec}(\bbF_q)}(\mathbf{E}_{\mathbf{V},\Omega})_0\rightarrow (\mathbf{E}_{\mathbf{V},\Omega})_0$ be the map induced by $\operatorname{spec}(\overline{\mathbb{F}_{q}})\rightarrow \operatorname{spec}(\bbF_q)$. We denote $\gamma^*$ by $\operatorname{egf}$.

By Lusztig's construction, the $\mathbf{G}_{\mathbf{V}}$-equivariant simple perverse sheaves without traceless summands are all of the form
\[
IC(\mathcal{O}_{\lambda}, \overline{\mathbb{Q}}_l)={j_{\lambda}}_{!*}(\overline{\mathbb{Q}}_l[\dim \mo_{\lambda}])
\]
for $\lambda$ invariant under the action of $a$. Moreover, $IC(\mathcal{O}_{\lambda}, \overline{\mathbb{Q}}_l)$ carries a naturally induced Weil structure
\[
(IC(\mathcal{O}_{\lambda}, \overline{\mathbb{Q}}_l), Fr)=(j_\lambda)_{!*}(\operatorname{egf}(\overline{\mathbb{Q}}_l[\dim \mo_{\lambda}](\frac{\dim\mo_{\lambda}}{2})|_{(\mo_{\lambda})_0})),
\]
together with
\[
\alpha_{\lambda} : a^* IC(\mathcal{O}_{\lambda}, \overline{\mathbb{Q}}_l) \to IC(\mathcal{O}_{\lambda}, \overline{\mathbb{Q}}_l),
\]
which is also a morphism of Weil structures by \cite{lusztig1998canonical}, satisfying $j_{\lambda}^*(\alpha_{\lambda})=\operatorname{id}$ and $\alpha_{\lambda}\circ a^*\alpha_{\lambda}\cdots (a^*)^{n-1}\alpha_{\lambda}=\operatorname{id}$. $(IC(\mathcal{O}_{\lambda}, \overline{\mathbb{Q}}_l),\alpha_{\lambda})$ is stable under Verdier duality.

Similarly, if we set $C_{\lambda}=\operatorname{egf}\bigl((j_{\lambda})_! \overline{\mathbb{Q}}_l[\dim \mo_{\lambda}](\frac{\dim\mo_{\lambda}}{2}) |_{(\mathcal{O}_{\lambda})_0}\bigr)$, the geometric realization of PBW basis in \cite{lan2025structurecoefficientsquantumgroups}, then we also obtain
\[
\psi_{\lambda} : a^* C_{\lambda} \to C_{\lambda},
\]
which is also a morphism of Weil structures satisfying $j_{\lambda}^*(\psi_{\lambda})$ is the canonical map and $\psi_{\lambda}\circ a^*\psi_{\lambda}\cdots (a^*)^{n-1}\psi_{\lambda}=\operatorname{id}$. 

Ignoring the Weil structure, for $IC(\mathcal{O}_{\lambda}, \overline{\mathbb{Q}}_l)$ supported on $\overline{\mathcal{O}_{\lambda}}$, let
\[
i_{\lambda} : \overline{\mathcal{O}_{\lambda}} - \mathcal{O}_{\lambda} \hookrightarrow \mathbf{E}_{\mathbf{V},\Omega}
\]
denote the inclusion. Then we have the following distinguished triangle:
\begin{equation}
\begin{tikzcd}
a^*C_{\lambda} \arrow[d, "\psi_{\lambda}"] \arrow[r] & {a^*IC(\mo_{\lambda},\overline{\bbQ_{l}})} \arrow[d, "\alpha_{\lambda}"] \arrow[r] & {a^*(i_{\lambda})_!(i_{\lambda})^*IC(\mo_{\lambda},\overline{\bbQ_{l}})} \arrow[d, "(i_{\lambda})_!(i_{\lambda})^*\alpha_{\lambda}"] \arrow[r, "+1"] & {} \\
C_{\lambda} \arrow[r]                                & {IC(\mo_{\lambda},\overline{\bbQ_{l}})} \arrow[r]                                  & {(i_{\lambda})_!(i_{\lambda})^*IC(\mo_{\lambda},\overline{\bbQ_{l}})} \arrow[r, "+1"]                                                                & {}
\end{tikzcd}
\end{equation}
We now consider the maps between $R(\nu)$-modules induced by $\mathrm{Ext}^{\bullet}(-, L)$. Since $\alpha_{\lambda}$, $\psi_{\lambda}$, and $(i_{\lambda})_! (i_{\lambda})^* \alpha_{\lambda}$ are morphisms of distinguished triangles, the induced maps $\hat{\alpha_{\lambda}}$, $\hat{\psi_{\lambda}}$, and $\widehat{(i_{\lambda})_! (i_{\lambda})^* \alpha_{\lambda}}$ yield morphisms between the corresponding short exact sequences:
\[
\begin{tikzcd}
a^* \mathrm{Ext}^{\bullet}(C_{\lambda}, L) \arrow[d, "\hat{\psi_{\lambda}}"] & 
a^* \mathrm{Ext}^{\bullet}(IC(\mathcal{O}_{\lambda}, \overline{\mathbb{Q}}_l), L) \arrow[d, "\hat{\alpha_{\lambda}}"] \arrow[l, "a^*f"] & 
a^* \mathrm{Ext}^{\bullet}((i_{\lambda})_!(i_{\lambda})^* IC(\mathcal{O}_{\lambda}, \overline{\mathbb{Q}}_l), L) \arrow[d, "\widehat{(i_{\lambda})_!(i_{\lambda})^* \alpha_{\lambda}}"] \arrow[l] \\
\mathrm{Ext}^{\bullet}(C_{\lambda}, L) & 
\mathrm{Ext}^{\bullet}(IC(\mathcal{O}_{\lambda}, \overline{\mathbb{Q}}_l), L) \arrow[l, "f"] & 
\mathrm{Ext}^{\bullet}((i_{\lambda})_!(i_{\lambda})^* IC(\mathcal{O}_{\lambda}, \overline{\mathbb{Q}}_l), L) \arrow[l]
\end{tikzcd}
\]

By \cite[Theorem 2.7]{cd088214-009e-3813-9a86-97f9a299a24c}, the dual standard module of $R(\nu)$ is
\[
\widetilde{K}_{\lambda} = \mathrm{Ext}^{\bullet}(C_{\lambda}, L),
\]
and $f$ is surjective. Therefore,
\[
f : (\mathrm{Ext}^{\bullet}(IC(\mathcal{O}_{\lambda}, \overline{\mathbb{Q}}_l), L), \hat{\alpha_{\lambda}}) \to (\mathrm{Ext}^{\bullet}(C_{\lambda}, L), \hat{\psi_{\lambda}})
\]
is a surjection of $R(\nu)\langle\mathbf{a}\rangle$-modules.
\begin{definition}
   The dual standard module of $R(\nu)\langle\mathbf{a}\rangle$ is defined as 
\[
(\widetilde{K}_{\lambda}, \gamma_{\lambda}) = \bigl(\mathrm{Ext}^{\bullet}(C_{\lambda}, L), \hat{\psi}_{\lambda}\bigr).
\]
\end{definition}
These dual standard modules correspond to the PBW basis.

Since $R(\nu)\langle\mathbf{a}\rangle$ has finite global dimension, we have $(\widetilde{K}_{\lambda}, \gamma_{\lambda}) \in \mathcal{P}_{\nu}$. Let
\[
(\mathbf{P}_{\lambda}, \hat{\phi_{\lambda}}) = (\mathrm{Ext}^{\bullet}(IC(\mathcal{O}_{\lambda}, \overline{\mathbb{Q}}_l), L), \hat{\alpha_{\lambda}}),
\]
and denote the corresponding simple module by $(\mathbf{L}_{\lambda}, \theta_{\lambda})$.

For the standard module of $R(\nu)$,
\[
K_{\lambda} = \mathbf{P}_{\lambda} \big/ \sum_{f \in \mathrm{Hom}_{R(\nu)\text{-}\mathrm{mod}}(\mathbf{P}_{\mu}, \mathbf{P}_{\lambda})_{>0}, \mu \leq \lambda} \mathrm{Im} f,
\]
observe that $\hat{\phi_{\lambda}}$ has degree zero and satisfies the following property: for any
\[
f \in \mathrm{Hom}_{R(\nu)\text{-}\mathrm{mod}}(\mathbf{P}, \mathbf{P}_{\lambda}),
\]
where $\mathbf{P}$ is either an indecomposable projective $R(\nu)$-module (corresponding to $\hat{\phi}$) or $\bigoplus_k (a^*)^k \mathbf{P}'$ (corresponding to $\hat{\phi_{\oplus}}$), we have
\[
\hat{\phi_{\lambda}}(\mathrm{Im} \mathbf{a}^*f)=\mathrm{Im}(\hat{\phi_{\lambda}}\circ \mathbf{a}^*f),f \in \mathrm{Hom}_{R(\nu)\text{-}\mathrm{mod}}(\mathbf{P}_{\mu}, \mathbf{P}_{\lambda})_{>0}, \mu \leq \lambda
\]
Since $\hat{\phi_{\lambda}}\circ \mathbf{a}^*f\in\ \mathrm{Hom}_{R(\nu)\text{-}\mathrm{mod}}(\mathbf{P}_{\mu}, \mathbf{P}_{\lambda})_{>0}$ for $\mu \leq \lambda$, this image is contained in
\[
\sum_{f \in \mathrm{Hom}_{R(\nu)\text{-}\mathrm{mod}}(\mathbf{P}_{\mu}, \mathbf{P}_{\lambda})_{>0}, \mu \leq \lambda} \mathrm{Im} f.
\]
Hence $\hat{\phi_{\lambda}}$ descends to a map on $K_{\lambda}$, which we denote by $\beta_{\lambda}$; the resulting object is the standard module, written $(K_{\lambda}, \beta_{\lambda}).$ 

Following Kato's method in \cite{cd088214-009e-3813-9a86-97f9a299a24c}, one can explicitly express $(\widetilde{K}_{\lambda}, \gamma_{\lambda})$ in terms of a projective resolution by $R(\nu)\langle\mathbf{a}\rangle$-modules consisting of projective and traceless summands. Denote by $\widetilde{\mathcal{D}^{b}_{m,(\mathbf{G}_{\mathbf{V}})_0}((\mathbf{E}_{\mathbf{V}})_0)}$ the bounded $(\mathbf{G}_{\mathbf{V}})_0$-equivariant mixed derived category on $(\mathbf{E}_{\mathbf{V}})_0$.

In Lemma \ref{leqq} and Corollary \ref{corll}, we assume that $(K,\phi)$ satisfies the following condition: for any $n$, ${}^p\tau^{\geq n}(K,\phi)$ is isomorphic to the mapping cone of ${}^p\tau^{\geq n+1}(K,\phi)[-1]$ and ${}^p\mathrm{H}^n((K,\phi))[-n]$.
\begin{lemma}\label{leqq}
Let $(\mathbf{E}_{\mathbf{V}})_0$ denote the $\mathbb{F}_q$-rational points of $\mathbf{E}_{\mathbf{V}}$. Let $(K,\phi) \in \widetilde{\mathcal{D}^{b}_{m,(\mathbf{G}_{\mathbf{V}})_0}((\mathbf{E}_{\mathbf{V}})_0)}$. For any $l \in \mathbb{Z}$, there exist $(F_{<l}K,\phi_{<l})$ and $(F_{\geq l}K,\phi_{\geq l})$ such that $F_{<l}K$ has weights $< l$ and $F_{\geq l}K$ has weights $\geq l$. Moreover, there exists a distinguished triangle:
\[
\operatorname{egf}(F_{<l}K,\phi_{<l}) \xrightarrow{\operatorname{egf}(f)} \operatorname{egf}(K,\phi) \xrightarrow{\operatorname{egf}(g)} \operatorname{egf}(F_{\geq l}K,\phi_{\geq l}) \xrightarrow{+1},
\]
meaning that $\operatorname{egf}(\phi_{<l})$, $\operatorname{egf}(\phi)$, and $\operatorname{egf}(\phi_{\geq l})$ are the morphisms between the triangle. Furthermore, there exists a distinguished triangle:
\[
F_{<l}K \xrightarrow{f} K \xrightarrow{g} F_{\geq l}K \xrightarrow{+1}.
\]
\end{lemma}

\begin{proof}
We consider the perverse $t$-structure on $\mathcal{D}^b_{m,(\mathbf{G}_{\mathbf{V}})_0}((\mathbf{E}_{\mathbf{V}})_0)$. Since $(K,\phi) \in \widetilde{\mathcal{D}^{b}_{m,(\mathbf{G}_{\mathbf{V}})_0}((\mathbf{E}_{\mathbf{V}})_0)}$, we have $\tau_{>m}K = 0$ for $m \gg 0$. Hence, we may prove the lemma by descending induction on $m$.

Assume that for $(\tau_{>m}K,\tau_{>m}\phi)$ there exist $(K_{>m}^{<l},\phi^{>m}_{<l}) := (F_{<l} \tau_{>m}K, \phi^{>m}_{<l})$ and $(K_{>m}^{\geq l},\phi^{>m}_{\geq l}) := (F_{\geq l} \tau_{>m}K, \phi^{>m}_{\geq l})$ such that the weights satisfy the requirements, and
\[
\operatorname{egf}(K_{>m}^{<l},\phi^{>m}_{<l}) \xrightarrow{\operatorname{egf}(f_m)} \operatorname{egf}(\tau_{>m}K,\tau_{>m}\phi) \xrightarrow{\operatorname{egf}(g_m)} \operatorname{egf}(K_{>m}^{\geq l},\phi^{>m}_{\geq l}) \xrightarrow{+1}
\]
is a distinguished triangle. Also,
\[
K_{>m}^{<l} \xrightarrow{f_m} \tau_{>m}K \xrightarrow{g_m} K_{>m}^{\geq l} \xrightarrow{+1}
\]
is a distinguished triangle.

This is clear for sufficiently large $m$, since then $\tau_{>m}K = 0$. By \cite[Theorem 5.4.16]{Pramod-2021}, $^{p}\mathrm{H}^m(K)[-m]$ admits a weight filtration with a subobject $^{p}\mathrm{H}^m_{<l}(K)[-m]$ of weight $< l$ and a quotient $\mathrm{H}^m_{\geq l} = (^{p}\mathrm{H}^m(K)/^{p}\mathrm{H}^m_{<l}(K))[-m]$ of weight $\geq l$. The map $^{p}\mathrm{H}^m(\phi)[-m]$ induces morphisms $\phi^m_{<l}$ and $\phi^m_{\geq l}$ on these subquotients.

We have the following diagram:
\[
\begin{tikzcd}
K_{>m}^{<l}[-1] \arrow[r, "f_m"] & \tau_{>m}K[-1] \arrow[r, "g_m"] \arrow[d] & K_{>m}^{\geq l}[-1] \\
^{p}\mathrm{H}^m_{<l}(K)[-m] \arrow[r] & ^{p}\mathrm{H}^m(K)[-m] \arrow[d] \arrow[r] & \mathrm{H}^m_{\geq l} \\
& \tau_{\geq m}K &
\end{tikzcd}
\]
\emph{(Diagram 1: Construction of the weight decomposition triangle)}

By \cite[Lemma 5.4.14(3)]{Pramod-2021}, we have:
\[
\mathrm{Hom}_{\mathcal{D}^b_{m,(\mathbf{G}_{\mathbf{V}})_0}((\mathbf{E}_{\mathbf{V}})_0)}(K_{>m}^{<l}[-1],\mathrm{H}^m_{\geq l})=0.
\]
Hence there is a commutative diagram in which each row and column is a distinguished triangle; it is obtained by taking the objects in the third row as mapping cones:
\[
\begin{tikzcd}
K_{>m}^{<l}[-1] \arrow[r, "f_m"] \arrow[d] & \tau_{>m}K[-1] \arrow[r, "g_m"] \arrow[d] & K_{>m}^{\geq l}[-1] \arrow[d] \\
^{p}\mathrm{H}^m_{<l}(K)[-m] \arrow[r] \arrow[d] & ^{p}\mathrm{H}^m(K)[-m] \arrow[d] \arrow[r] & \mathrm{H}^m_{\geq l} \arrow[d] \\
K_{>m-1}^{<l} \arrow[r, "f_{m-1}"] & \tau_{\geq m}K \arrow[r, "g_{m-1}"] & K_{>m-1}^{\geq l}
\end{tikzcd}
\]
\emph{(Diagram 2: Construction of the weight decomposition triangle II)}

By \cite[Lemma 5.4.14(1)]{Pramod-2021}, we have:
\[
\operatorname{egf}\left(\mathrm{Hom}_{\mathcal{D}^b_{m,(\mathbf{G}_{\mathbf{V}})_0}((\mathbf{E}_{\mathbf{V}})_0)}(K_{>m}^{<l}[-1],\mathrm{H}^m_{\geq l}[-1])\right)=0
\]
and
\[
\operatorname{egf}\left(\mathrm{Hom}_{\mathcal{D}^b_{m,(\mathbf{G}_{\mathbf{V}})_0}((\mathbf{E}_{\mathbf{V}})_0)}(K_{>m}^{<l},\mathrm{H}^m_{\geq l})\right)=0.
\]
Under the notation of \cite[Remark 5.3.5]{Pramod-2021}, with $\operatorname{egf}(\Hom_{\mathcal{D}^b_{m,(\mathbf{G}_{\mathbf{V}})_0}((\mathbf{E}_{\mathbf{V}})_0)}(F,G))\cong \underline{\Hom}_{\mathcal{D}^b_{m,(\mathbf{G}_{\mathbf{V}})_0}((\mathbf{E}_{\mathbf{V}})_0)}(F,G)^{\operatorname{Fr}},$ $$\operatorname{id}-\operatorname{Fr:}\underline{\mathrm{Hom}}_{\mathcal{D}^b_{m,(\mathbf{G}_{\mathbf{V}})_0}((\mathbf{E}_{\mathbf{V}})_0)}(K_{>m}^{<l}[-1],\mathrm{H}^m_{\geq l})\rightarrow\underline{\mathrm{Hom}}_{\mathcal{D}^b_{m,(\mathbf{G}_{\mathbf{V}})_0}((\mathbf{E}_{\mathbf{V}})_0)}(K_{>m}^{<l}[-1],\mathrm{H}^m_{\geq l})$$ and $$\operatorname{id}-\operatorname{Fr:}\underline{\mathrm{Hom}}_{\mathcal{D}^b_{m,(\mathbf{G}_{\mathbf{V}})_0}((\mathbf{E}_{\mathbf{V}})_0)}(K_{>m}^{<l},\mathrm{H}^m_{\geq l})\rightarrow\underline{\mathrm{Hom}}_{\mathcal{D}^b_{m,(\mathbf{G}_{\mathbf{V}})_0}((\mathbf{E}_{\mathbf{V}})_0)}(K_{>m}^{<l},\mathrm{H}^m_{\geq l})$$ are isomorphic maps.

Hence, by \cite[Remark 5.3.5]{Pramod-2021}, and applying $\underline{\mathrm{Hom}}_{\mathcal{D}^b_{m,(\mathbf{G}_{\mathbf{V}})_0}((\mathbf{E}_{\mathbf{V}})_0)}(-,\mathrm{H}^m_{\geq l})$ and $\underline{\mathrm{Hom}}_{\mathcal{D}^b_{m,(\mathbf{G}_{\mathbf{V}})_0}((\mathbf{E}_{\mathbf{V}})_0)}(K_{>m}^{<l}[-1],-)$ to the triangles, we have:
\[
\operatorname{egf}\left(\mathrm{Hom}_{\mathcal{D}^b_{m,(\mathbf{G}_{\mathbf{V}})_0}((\mathbf{E}_{\mathbf{V}})_0)}(K_{>m}^{\geq l}[-1],\mathrm{H}^m_{\geq l})\right)
\cong
\operatorname{egf}\left(\mathrm{Hom}_{\mathcal{D}^b_{m,(\mathbf{G}_{\mathbf{V}})_0}((\mathbf{E}_{\mathbf{V}})_0)}(\tau_{>m}K[-1],\mathrm{H}^m_{\geq l})\right)
\]
and
\[
\operatorname{egf}\left(\mathrm{Hom}_{\mathcal{D}^b_{m,(\mathbf{G}_{\mathbf{V}})_0}((\mathbf{E}_{\mathbf{V}})_0)}(K_{>m}^{<l}[-1],^{p}\mathrm{H}^m_{<l}(K)[-m])\right)
\cong
\operatorname{egf}\left(\mathrm{Hom}_{\mathcal{D}^b_{m,(\mathbf{G}_{\mathbf{V}})_0}((\mathbf{E}_{\mathbf{V}})_0)}(K_{>m}^{<l}[-1],^{p}\mathrm{H}^m(K)[-m])\right).
\]

Therefore, after applying the $\operatorname{egf}$ functor, the diagram can be written as follows:
\[
\begin{tikzcd}
\operatorname{egf}(K_{>m}^{<l},\phi^{>m}_{<l}) \arrow[r, "\operatorname{egf}(f_m)"] \arrow[d] & \operatorname{egf}(\tau_{>m}K,\tau_{>m}\phi) \arrow[r, "\operatorname{egf}(g_m)"] \arrow[d] & \operatorname{egf}(K_{>m}^{\geq l},\phi^{>m}_{\geq l}) \arrow[d] \\
\operatorname{egf}(^{p}\mathrm{H}^m_{<l}(K)[-m],\phi_{<l}^m) \arrow[r] \arrow[d] & \operatorname{egf}(^{p}\mathrm{H}^m(K)[-m], ^{p}\mathrm{H}^m(\phi)[-m]) \arrow[d] \arrow[r] & \operatorname{egf}(\mathrm{H}^m_{\geq l},\phi^m_{\geq l}) \arrow[d] \\
\operatorname{egf}(K_{>m-1}^{<l},\phi^{>m-1}_{<l}) \arrow[r, "\operatorname{egf}(f_{m-1})"] & \operatorname{egf}(\tau_{>m-1}K,\tau_{>m-1}\phi) \arrow[r, "\operatorname{egf}(g_{m-1})"] & \operatorname{egf}(K_{>m}^{\geq l+1},\phi^{>m}_{\geq l+1})
\end{tikzcd}
\]
\emph{(Diagram 3: Construction of the weight decomposition triangle III)}

Hence the induction step holds at level $m-1$.

By induction, when $m$ is sufficiently small, $\tau_{>m}K \cong K$, completing the proof.
\end{proof}
The construction in the proof of the preceding lemma immediately gives the following corollary.

\begin{corollary}\label{corll}
For any $l' < l$, we have
\[
\operatorname{egf}(F_{<l'} F_{<l} K, (\phi_{<l})_{<l'}) \cong \operatorname{egf}(F_{<l'} K, \phi_{<l'}),
\]
and
\[
\operatorname{egf}(F_{\geq l} F_{\geq l'} K, (\phi_{\geq l'})_{\geq l}) \cong \operatorname{egf}(F_{\geq l} K, \phi_{\geq l}).
\]
\end{corollary}

Therefore, there exists a distinguished triangle
\[
\operatorname{egf}(F_{<l} F_{\geq l-1} K, (\phi_{\geq l-1})_{<l}) \to \operatorname{egf}(F_{\geq l-1} K, \phi_{\geq l-1}) \to \operatorname{egf}(F_{\geq l} F_{\geq l-1} K, (\phi_{\geq l-1})_{\geq l}) \xrightarrow{+1}.
\]
Since $F_{<l} F_{\geq l-1} K$ is a pure perverse sheaf of weight $l-1$, by \cite[Theorem 5.4.5]{BBD}, $\operatorname{egf}(F_{<l} F_{\geq l-1} K)$ is semisimple. We denote
\[
\operatorname{gr}_{l-1} K = \operatorname{egf}(F_{<l} F_{\geq l-1} K).
\]
By \cite[11.1.3]{lusztig2010introduction},
\[
(\operatorname{gr}_{l-1} K ,(\phi_{\geq l-1})_{<l})\cong \bigoplus_{P\text{ is simple perverse, }a^*P\cong P}(P,\phi)\oplus \bigoplus_{(Q,\psi)\text{ is traceless}}(Q,\psi).
\]
Since
\[
C_{\lambda} = \operatorname{egf}\bigl((j_{\lambda})_! \overline{\mathbb{Q}}_l |_{(\mathcal{O}_{\lambda})_0}\bigr)[\dim \mathcal{O}_{\lambda}],
\]
and $\dim\Hom_{\cD^b_{\bfG_{\bfV}}}(\bfE_{\bfV})(a^*C_{\lambda},C_{\lambda})=1$. Thus $(C_{\lambda},\psi)$ is isomorphic to the mapping cone of ${}^p\mathrm{H}^m((C_{\lambda},\psi))$, and we can apply Lemma \ref{leqq} and Corollary \ref{corll} to $(C_{\lambda},\psi)$.
By the decomposition theorem in \cite{BBD}, it follows that $\operatorname{gr}_l C_{\lambda}$ is semisimple. We define
\[
Q(C_{\lambda})_l = \mathrm{Ext}^{\bullet}(\operatorname{gr}_l C_{\lambda}, L),
\]
and write $h_{l-1} := \operatorname{egf}\bigl((\phi_{\geq l-1})_{<l}\bigr)$, which naturally induces
\[
\hat{h}_{l-1} : \mathbf{a}^* Q(C_{\lambda})_{l-1} \to Q(C_{\lambda})_{l-1}.
\]

Using the method of Kato \cite{cd088214-009e-3813-9a86-97f9a299a24c}, we also obtain the following theorem.

\begin{theorem}
The complex $(Q(C_{\lambda}), d),$ where $d_i : Q(C_{\lambda})_i \to Q(C_{\lambda})_{i+1}$ constructed above satisfies
\[
\mathrm{H}^{\bullet}\bigl((Q(C_{\lambda}), d)\bigr) = \widetilde{K}_{\lambda},
\]
and
\[
h_{i+1} \circ \mathbf{a}^* d_i = d_i \circ h_i.
\]
Moreover, $(Q(C_{\lambda}), d)$ is a projective resolution of $\widetilde{K}_{\lambda}$, and the collection $\{h_i\}$ induces the map $\gamma_{\lambda}$.
\end{theorem}
\begin{proof}
For brevity, we denote $\operatorname{egf} F_{\geq l-1} (j_{\lambda})_! \overline{\mathbb{Q}}_l |_{(\mathcal{O}_{\lambda})_0} [\dim \mathcal{O}_{\lambda}]$ by $F_{\geq l-1} C_{\lambda}$, and denote the induced map by $h_{\geq l-1}$.
Consider the morphism of distinguished triangles:
\[
(\operatorname{gr}_{l-1}C_{\lambda},h_{l-1}) \to 
(F_{\geq l-1} C_{\lambda}, h_{\geq l-1}) \to 
(F_{\geq l} C_{\lambda}, h_{\geq l}) \xrightarrow{+1}.
\]
Applying $\mathrm{Ext}^{\bullet}(-, L)$, we obtain the following morphisms between short exact sequences:
\[
\begin{tikzcd}
\mathbf{a}^* \mathrm{Ext}^{\bullet}(F_{\geq l+1} C_{\lambda}, L) \arrow[d, "\hat{h}_{\geq l}"] \arrow[r] &
\mathbf{a}^* \mathrm{Ext}^{\bullet}(F_{\geq l} C_{\lambda}, L) \arrow[d, "\hat{h}_{\geq l-1}"] \arrow[r] &
\mathbf{a}^* Q(C_{\lambda})_l \arrow[d, "\hat{h}_l"] \arrow[r, "a^* \delta"] &
\mathbf{a}^* \mathrm{Ext}^{\bullet}(F_{\geq l+1} C_{\lambda}, L) \langle -1 \rangle \\
\mathrm{Ext}^{\bullet}(F_{\geq l+1} C_{\lambda}, L) \arrow[r] &
\mathrm{Ext}^{\bullet}(F_{\geq l} C_{\lambda}, L) \arrow[r] &
Q(C_{\lambda})_l \arrow[r, "\delta"] &
\mathrm{Ext}^{\bullet}(F_{\geq l+1} C_{\lambda}, L) \langle -1 \rangle
\end{tikzcd}
\]

We prove by descending induction on $l$ that
\[
\mathrm{H}^{\bullet}((Q(C_{\lambda})_{\geq l}, d)) \cong \mathrm{Ext}^{\bullet}(F_{\geq l} C_{\lambda}, L),
\]
and that for $i \geq l$, $h_{i+1} \circ \mathbf{a}^* d_i = d_i \circ h_i$. For sufficiently large $l$, the modules involved are zero, so the statement holds in the initial case.

Assume it holds for $l+1$. For $l$, consider the connecting map $\delta$ in the short exact sequence above, which has degree one. The module $Q(C_{\lambda})_l$ is pure of weight $-l$, while $\mathrm{Ext}^{\bullet}(F_{\geq l+1} C_{\lambda}, L)$ has weights $< -l$. By the inductive hypothesis, $\mathrm{H}^{\bullet}((Q(C_{\lambda})_{\geq l+1}, d))$ also has weights $< -l$.

Thus, $\ker \delta$ has weight $-l$ and $\operatorname{coker} \delta \langle 1 \rangle$ has weight $< -l$. Moreover, $\ker \delta$ and $\operatorname{coker} \delta \langle 1 \rangle$ are, respectively, an $R(\nu)$-quotient module and a submodule of $\mathrm{Ext}^{\bullet}(F_{\geq l+1} C_{\lambda}, L)$. Dimension-vector computations show
\[
\mathrm{Ext}^{\bullet}(F_{\geq l+1} C_{\lambda}, L) \cong \ker \delta \oplus \operatorname{coker} \delta \langle 1 \rangle.
\]
Since $\ker d_{l+1}$ is the $-(l+1)$-weight summand of $\mathrm{H}^{\bullet}((Q(C_{\lambda})_{\geq l+1}, d))$, we have $\operatorname{Im} \delta \subset \ker d_{l+1}$. We thus define
\[
d_l : Q(C_{\lambda})_l \to Q(C_{\lambda})_{l+1} \langle -1 \rangle
\]
as the shift of $\delta$ to degree zero.

Moreover,
\[
\mathrm{Ext}^{\bullet}(F_{\geq l+1} C_{\lambda}, L) \cong \ker d_{l+1} \oplus_{s>l+1} \mathrm{H}^s((Q(C_{\lambda})_{\geq l+1}, d)),
\]
so via the short exact sequence, we find
\[
\bigoplus_{s>l} \mathrm{H}^s((Q(C_{\lambda})_{\geq l}, d)) \cong \operatorname{coker} \delta \langle 1 \rangle.
\]
Thus $\mathrm{H}^{\bullet}\bigl((Q(C_{\lambda}), d)\bigr) = \widetilde{K}_{\lambda}$. By the pointwise purity of $C_{\lambda}$, this gives a projective resolution, as in \cite[Proposition 2.8]{cd088214-009e-3813-9a86-97f9a299a24c}.
Since $d_l$ is obtained by shifting the connecting morphism, it is clear that
\[
h_{i+1} \circ \mathbf{a}^* d_i = d_i \circ h_i.
\]
\end{proof}

Thus, from the construction of $(Q(C_{\lambda}), d)$ and $\hat{h}_l$, it is evident that, as an $R(\nu)\langle\mathbf{a}\rangle$-module, each direct summand in this resolution is either projective or traceless. Moreover, when $a =\mathbf{a}= \mathrm{id}$, these results recover Kato's original results.

We now compute the value of the pairing $\langle -, - \rangle_{\nu}$ between $[\mathbb{D}(K_{\mu}, \beta_{\mu})]$ (the dual in $\mathcal{L}_{\nu}$) and $[(\widetilde{K}_{\lambda}, \gamma_{\lambda})]$, where $\mu, \lambda$ are $a$-stable orbits in $\Lambda$.

\begin{proposition}
    For any $a$-stable orbits $\mu, \lambda$ in $\Lambda$, we have
    \[
        \langle [(\widetilde{K}_{\lambda}, \gamma_{\lambda})], [\mathbb{D}(K_{\mu}, \beta_{\mu})] \rangle_{\nu} = \delta_{\lambda, \mu}.
    \]
\end{proposition}

\begin{proof}
This follows from Proposition 3.9 in \cite{cd088214-009e-3813-9a86-97f9a299a24c}, where Kato proved that

\[
    \Ext^k_{R(\nu)\text{-mod}}(\widetilde{K}_{\lambda}, \mathbb{D}K_{\mu}) = 0 \quad \text{for } k \neq 0 \text{ or } \mu \neq \lambda,
\]
and
\[
    \Hom_{R(\nu)\text{-mod}}(\widetilde{K}_{\lambda}, \mathbb{D}K_{\lambda}) \cong \Hom_{R(\nu)\text{-mod}}(\widetilde{K}_{\lambda}, \mathbb{D}\mathbf{L}_{\lambda}) \cong k.
\]

Since $(\mathbf{P}_{\lambda}, \hat{\phi}_{\lambda})$ is self-dual, it follows that $(\mathbf{L}_{\lambda}, \theta_{\lambda})$ is also self-dual. Therefore, the induced action of $a_{\gamma_{\lambda}, \beta_{\mu}}$ on $\Hom_{R(\nu)\text{-mod}}(\widetilde{K}_{\lambda}, \mathbb{D}K_{\lambda})$ coincides with that of $a_{\gamma_{\lambda}, \theta_{\lambda}}$ on $\Hom_{R(\nu)\text{-mod}}(\widetilde{K}_{\lambda}, \mathbb{D}\mathbf{L}_{\lambda})$, both being the identity. Hence,
\[
    \langle [(\widetilde{K}_{\lambda}, \gamma_{\lambda})], [\mathbb{D}(K_{\mu}, \beta_{\mu})] \rangle_{\nu} = \delta_{\lambda, \mu}.
\]
\end{proof}
We already have the dual bases $\{[(\mathbf{P}_{\lambda}, \hat{\phi}_{\lambda})] \mid \lambda \in \Lambda,\ a\lambda=\lambda \}$ in $\mathbb{K}(\mathcal{P}_{\nu})$ and $\{[(\mathbf{L}_{\mu}, \theta_{\mu})] \mid \mu \in \Lambda,\ a\mu=\mu \}$ in $\mathbb{K}(\mathcal{L}_{\nu})$. Having obtained another pair of dual bases, $\{[(\widetilde{K}_{\lambda}, \gamma_{\lambda})]\}$ and $\{[\mathbb{D}(K_{\mu}, \beta_{\mu})]\}$, we may express them as follows:

In $\mathbb{K}(\mathcal{P}_{\nu})$, we write:
\[
[(\widetilde{K}_{\lambda}, \gamma_{\lambda})] = \sum_{\mu \in \Lambda,\ a\text{-invariant}} [(\widetilde{K}_{\lambda}, \gamma_{\lambda}):(\mathbf{P}_{\mu}, \hat{\phi}_{\mu})] \cdot [(\mathbf{P}_{\mu}, \hat{\phi}_{\mu})],
\]
\[
[(\mathbf{P}_{\mu}, \hat{\phi}_{\mu})] = \sum_{\lambda \in \Lambda,\ a\text{-invariant}} [(\mathbf{P}_{\mu}, \hat{\phi}_{\mu}):(\widetilde{K}_{\lambda}, \gamma_{\lambda})] \cdot [(\widetilde{K}_{\lambda}, \gamma_{\lambda})].
\]

Similarly, in $\mathbb{K}(\mathcal{L}_{\nu})$, we write
\[
[\mathbb{D}(K_{\lambda}, \beta_{\lambda})] = \sum_{\mu \in \Lambda,\ a\text{-invariant}} [\mathbb{D}(K_{\lambda}, \beta_{\lambda}):(\mathbf{L}_{\mu}, \theta_{\mu})] \cdot [(\mathbf{L}_{\mu}, \theta_{\mu})],
\]
\[
[(\mathbf{L}_{\mu}, \theta_{\mu})] = \sum_{\lambda \in \Lambda,\ a\text{-invariant}} [(\mathbf{L}_{\mu}, \theta_{\mu}):\mathbb{D}(K_{\lambda}, \beta_{\lambda})] \cdot [\mathbb{D}(K_{\lambda}, \beta_{\lambda})],
\]

where $[(\widetilde{K}_{\lambda}, \gamma_{\lambda}):(\mathbf{P}_{\mu}, \hat{\phi}_{\mu})],\ [\mathbb{D}(K_{\lambda}, \beta_{\lambda}):(\mathbf{L}_{\mu}, \theta_{\mu})] \in \mathcal{O}[t, t^{-1}]$.

We define the bar involution on $\mathcal{O}[t, t^{-1}]$ by
\[
\bar{t} = t^{-1},\quad \bar{\zeta} = \zeta^{-1},\quad \bar{n} = n \quad \text{for } n \in \mathbb{Z}.
\]
We call the basis $[(\mathbf{P}_{\mu}, \hat{\phi}_{\mu})]$ the signed canonical basis, always choosing $\phi_{\mu}=\alpha_{\mu}$, and we call $[(\widetilde{K}_{\lambda}, \gamma_{\lambda})]$ the PBW basis.
\begin{theorem}
    For any $a$-invariant orbits $\lambda,\mu \in \Lambda$, we have:
    \[
        [(\mathbf{P}_{\mu}, \hat{\phi}_{\mu}):(\widetilde{K}_{\lambda}, \gamma_{\lambda})] = [(K_{\lambda}, \beta_{\lambda}):(\mathbf{L}_{\mu}, \theta_{\mu})].
    \]
    In particular, these coefficients lie in $\mathcal{O}[t, t^{-1}]$.
\end{theorem}

\begin{proof}
    Consider the matrices
    \[
    A = \left( [(\mathbf{P}_{\mu}, \hat{\phi}_{\mu}):(\widetilde{K}_{\lambda}, \gamma_{\lambda})] \right)_{\mu,\lambda} = \left( [(\widetilde{K}_{\lambda}, \gamma_{\lambda}):(\mathbf{P}_{\mu}, \hat{\phi}_{\mu})] \right)_{\lambda,\mu}^{-1},
    \]
    and
    \[
    B = \left( [(\mathbf{L}_{\mu}, \theta_{\mu}):\mathbb{D}(K_{\lambda}, \beta_{\lambda})] \right)_{\lambda,\mu} = \left( [\mathbb{D}(K_{\lambda}, \beta_{\lambda}):(\mathbf{L}_{\mu}, \theta_{\mu})] \right)_{\mu,\lambda}^{-1}.
    \]

    From the pairing, we deduce:
    \[
    A^{-1} \cdot \overline{B^{-1}} = \operatorname{id},
    \]
    and since
    \[
    [\mathbb{D}(K_{\lambda}, \beta_{\lambda}):(\mathbf{L}_{\mu}, \theta_{\mu})] = \overline{[(K_{\lambda}, \beta_{\lambda}):(\mathbf{L}_{\mu}, \theta_{\mu})]},
    \]
    it follows that:
    \[
    [(\mathbf{P}_{\mu}, \hat{\phi}_{\mu}):(\widetilde{K}_{\lambda}, \gamma_{\lambda})] = [(K_{\lambda}, \beta_{\lambda}):(\mathbf{L}_{\mu}, \theta_{\mu})].
    \]
\end{proof}

\begin{remark}
    As shown in \cite{cd088214-009e-3813-9a86-97f9a299a24c}, if $\mu < \lambda$, then the filtration of $K_{\lambda}$ contains no copy of $\mathbf{L}_{\mu}$ and contains exactly one copy of $\mathbf{L}_{\lambda}$ in degree zero. Since any simple summand in the filtration of $(K_{\lambda}, \beta_{\lambda})$ must be of the form $(\mathbf{L}_{\mu}, \zeta^r \theta_{\mu})$ or a trace-zero element, and if such $(\mathbf{L}_{\mu},\zeta^r \theta_{\mu})$ appears in the filtration of $K_{\lambda}$, it follows that $\mu \geq \lambda$.

    Furthermore, since $(\mathbf{L}_{\lambda}, \theta_{\lambda})$ appears as the top of $(K_{\lambda}, \beta_{\lambda})$, we have:
    \[
    [(K_{\lambda}, \beta_{\lambda}):(\mathbf{L}_{\lambda}, \theta_{\lambda})] = 1.
    \]
    Therefore, the matrix $\left( [(K_{\lambda}, \beta_{\lambda}):(\mathbf{L}_{\mu}, \theta_{\mu})] \right)_{\lambda,\mu}$ is upper triangular over $\mathcal{O}[t, t^{-1}]$ with $1$'s along the diagonal. In particular, for $\mu < \lambda$, we have:
    \[
    [(\mathbf{P}_{\mu}, \hat{\phi}_{\mu}):(\widetilde{K}_{\lambda}, \gamma_{\lambda})] = 0,
    \]
    and for $\mu = \lambda$:
    \[
    [(\mathbf{P}_{\lambda}, \hat{\phi}_{\lambda}):(\widetilde{K}_{\lambda}, \gamma_{\lambda})] = 1.
    \]
\end{remark}
Thus, if we denote
\[
p_{\mu,\lambda}(t)=[(\mathbf{P}_{\mu}, \hat{\phi}_{\mu}):(\widetilde{K}_{\lambda}, \gamma_{\lambda})]\in \mathcal{O}[t, t^{-1}],
\]
then in $\mk(\mathcal{P}_{\nu})$ we have
\[
[(\mathbf{P}_{\mu}, \hat{\phi}_{\mu})]=\sum_{\lambda\in \Lambda}p_{\mu,\lambda}(t)[(\widetilde{K}_{\lambda}, \gamma_{\lambda})].
\]
This gives the relation between the PBW basis and the canonical basis, and the matrix $(p_{\mu,\lambda}(t))_{\mu,\lambda\in \Lambda}$ is upper triangular with diagonal entries equal to one.
\begin{remark}
    When $a=\operatorname{id},$ our results recovers Kato's result \cite{cd088214-009e-3813-9a86-97f9a299a24c}, and when $a$ is not trivial, the positivity of the coeifficients of $p_{\mu,\lambda}(t)$ is lost.
\end{remark}
\section*{Acknowledgements}
The author would like to thank Professor Jie Xiao for helpful discussions. This work was supported by the National Natural Science Foundation of China(Grant No. 12471030).
\end{spacing}

\bibliography{ref}

@incollection {OSNotes,
    AUTHOR = {Schiffmann, Olivier},
     TITLE = {Lectures on canonical and crystal bases of {H}all algebras},
 BOOKTITLE = {Geometric methods in representation theory. {II}},
    SERIES = {S\'{e}min. Congr.},
    VOLUME = {24},
     PAGES = {143--259},
 PUBLISHER = {Soc. Math. France, Paris},
      YEAR = {2012},
   MRCLASS = {17B37 (14D23 14D24 14L30)},
  MRNUMBER = {3202708},
MRREVIEWER = {Kailash C. Misra},
}

@incollection {BBD,
    AUTHOR = {Be\u{\i}linson, A. A. and Bernstein, J. and Deligne, P.},
     TITLE = {Faisceaux pervers},
 BOOKTITLE = {Analysis and topology on singular spaces, {I} ({L}uminy,
              1981)},
    SERIES = {Ast\'{e}risque},
    VOLUME = {100},
     PAGES = {5--171},
 PUBLISHER = {Soc. Math. France, Paris},
      YEAR = {1982},
   MRCLASS = {32C38},
  MRNUMBER = {751966},
MRREVIEWER = {Zoghman Mebkhout},
}

@book{lusztig2010introduction,
	author = {Lusztig, George},
	isbn = {0-8176-3712-5},
	mrclass = {17B37 (16W30 17-02 17B35 81R50)},
	mrnumber = {1227098},
	mrreviewer = {Jie\ Du},
	pages = {xii+341},
	publisher = {Birkh\"{a}user Boston, Inc., Boston, MA},
	series = {Progress in Mathematics},
	title = {Introduction to quantum groups},
	volume = {110},
	year = {1993}}

@book{Pramod-2021,
	author = {Achar, Pramod N.},
	doi = {10.1090/surv/258},
	isbn = {978-1-4704-5597-2},
	mrclass = {32-02 (14F08 17B08 17B37 20G05 32S60)},
	mrnumber = {4337423},
	pages = {xii+562},
	publisher = {American Mathematical Society, Providence, RI},
	series = {Mathematical Surveys and Monographs},
	title = {Perverse sheaves and applications to representation theory},
	url = {https://mathscinet.ams.org/mathscinet-getitem?mr=4337423},
	volume = {258},
	year = {2021}}

@article{lusztig1998canonical,
	author = {Lusztig, George},
	journal = {Representation theories and algebraic geometry},
	pages = {365--399},
	publisher = {Springer},
	title = {Canonical bases and {H}all algebras},
	year = {1998}}

@article{cd088214-009e-3813-9a86-97f9a299a24c,
 ISSN = {00029327, 10806377},
 URL = {http://www.jstor.org/stable/44508995},
 author = {Syu Kato},
 journal = {American Journal of Mathematics},
 number = {3},
 pages = {567--615},
 publisher = {The Johns Hopkins University Press},
 title = {AN ALGEBRAIC STUDY OF EXTENSION ALGEBRAS},
 urldate = {2025-02-19},
 volume = {139},
 year = {2017}
}

@article{https://doi.org/10.1112/jlms.12218,
author = {McNamara, Peter J.},
title = {Folding {KLR} algebras},
journal = {Journal of the London Mathematical Society},
volume = {100},
number = {2},
pages = {447-469},
doi = {https://doi.org/10.1112/jlms.12218},
url = {https://londmathsoc.onlinelibrary.wiley.com/doi/abs/10.1112/jlms.12218},
eprint = {https://londmathsoc.onlinelibrary.wiley.com/doi/pdf/10.1112/jlms.12218},
year = {2019}
}

@article{VaragnoloVasserot+2011+67+100,
url = {https://doi.org/10.1515/crelle.2011.068},
title = {Canonical bases and {KLR}-algebras},
title = {},
author = {M. Varagnolo and E. Vasserot},
pages = {67--100},
volume = {2011},
number = {659},
journal = {Journal für die reine und angewandte Mathematik},
doi = {doi:10.1515/crelle.2011.068},
year = {2011},
lastchecked = {2025-02-19}
}

@article{lusztig1990canonical,
  title={Canonical bases arising from quantized enveloping algebras},
  author={Lusztig, George},
  journal={Journal of the American Mathematical Society},
  volume={3},
  number={2},
  pages={447--498},
  year={1990},
  publisher={JSTOR}
}

@article{REITEN1985224,
title = {Skew group algebras in the representation theory of artin algebras},
journal = {Journal of Algebra},
volume = {92},
number = {1},
pages = {224-282},
year = {1985},
issn = {0021-8693},
doi = {https://doi.org/10.1016/0021-8693(85)90156-5},
url = {https://www.sciencedirect.com/science/article/pii/0021869385901565},
author = {Idun Reiten and Christine Riedtmann}
}

@article{KLR,
    author = {Mikhail Khovanov and Aaron D. Lauda},
    title = {A diagrammatic approach to categorification of quantum groups I},
    journal = {Represent. Theory},
    volume={13},
    pages={309-347},
    year = {2009},
    doi ={https://doi.org/10.1090/S1088-4165-09-00346-X},
}

@article{doi:10.1142/S1005386712000247,
author = {Rouquier, Rapha\"{e}l},
title = {Quiver {Hecke} Algebras and 2-{Lie} Algebras},
journal = {Algebra Colloquium},
volume = {19},
number = {02},
pages = {359-410},
year = {2012},
doi = {10.1142/S1005386712000247},

URL = { 
    
        https://doi.org/10.1142/S1005386712000247
    
    

},
eprint = { 
    
        https://doi.org/10.1142/S1005386712000247
    
    

}
}

@misc{lan2025structurecoefficientsquantumgroups,
      title={Structure coefficients for quantum groups}, 
      author={Yixin Lan and Yumeng Wu and Jie Xiao},
      year={2025},
      eprint={2510.25575},
      archivePrefix={arXiv},
      primaryClass={math.RT},
      url={https://arxiv.org/abs/2510.25575}, 
}

@article{Bi2024,   
author  = {Bi, Yingjin},   
title   = {The product of simple modules over {KLR} algebras and quiver {G}rassmannians}, 
journal = {Representation Theory},  
volume  = {28}, 
number  = {17},  
pages   = {552--592},  
year    = {2024},  
doi     = {10.1090/ert/680} 
}

@article{MA202460,
title = {{Foldings of KLR algebras}},
journal = {Journal of Algebra},
volume = {639},
pages = {60-98},
year = {2024},
issn = {0021-8693},
doi = {https://doi.org/10.1016/j.jalgebra.2023.09.035},
url = {https://www.sciencedirect.com/science/article/pii/S002186932300501X},
author = {Ying Ma and Toshiaki Shoji and Zhiping Zhou}
}

@misc{rouquier20082kacmoodyalgebras,
      title={{2-Kac-Moody algebras}}, 
      author={Raphael Rouquier},
      year={2008},
      eprint={0812.5023},
      archivePrefix={arXiv},
      primaryClass={math.RT},
      url={https://arxiv.org/abs/0812.5023}, 
      note={arXiv:0812.5023}
}

\end{document}